\documentclass[12pt]{amsart}
\pdfoutput=1
\usepackage{mathtools}
\usepackage{amsmath}
\usepackage{amsthm}
\usepackage{amssymb}
\usepackage{amsbsy}
\usepackage{amstext}
\usepackage{amsopn}
\usepackage{enumerate}
\usepackage{xcolor}
\usepackage{graphicx}
\usepackage{microtype}
\usepackage[margin=1in,marginparwidth=0.8in, marginparsep=0.1in]{geometry}
\usepackage[pagebackref, bookmarks=true, bookmarksopen=true, bookmarksdepth=3,bookmarksopenlevel=2, colorlinks=true, linkcolor=blue, citecolor=blue, filecolor=blue, menucolor=blue, urlcolor=blue]{hyperref}
\usepackage{newtxtext} 
\usepackage{stmaryrd}
\usepackage{accents}
\usepackage{bbm}
\usepackage{tikz}
\usetikzlibrary{calc,decorations.markings,arrows}
\tikzset{
  singlearrowreversed/.style={draw=none,postaction={decorate,decoration={markings,mark=at position 0.5 with {\arrowreversed[xshift=-2pt]{stealth}}}}},
  singlearrow/.style={draw=none,postaction={decorate,decoration={markings,mark=at position 0.5 with {\arrow[xshift=2pt]{stealth}}}}},
  doublearrow/.style={draw=none,postaction={decorate,decoration={markings,mark=at position 0.25 with {\arrow{stealth}},mark=at position 0.85 with {\arrow{stealth}}}}},
  triplearrow/.style={draw=none,postaction={decorate,decoration={markings,mark=at position 0.1 with {\arrow{stealth}},mark=at position 0.5 with {\arrow{stealth}},mark=at position 0.9 with {\arrow{stealth}}}}},
  dots/.style={draw=none,postaction={decorate,decoration={markings,mark=at position 0.25 with {\draw[fill=black] (0,0) circle (0.75pt);},mark=at position 0.5 with {\draw[fill=black] (0,0) circle (0.75pt);},mark=at position 0.75 with {\draw[fill=black] (0,0) circle (0.75pt);}}}}
}

\newcommand{\newword}[1]{\textbf{\emph{#1}}}

\DeclareMathOperator{\corank}{corank}
\DeclareMathOperator{\Ext}{Ext}
\DeclareMathOperator{\Hom}{Hom}
\newcommand{\Bdp}{\widetilde{B}_{dp}}
\newcommand{\bfc}{\mathbf{c}}
\newcommand{\bfg}{\mathbf{g}}
\newcommand{\Bpr}{\widetilde{B}_{pr}}
\newcommand{\cA}{\mathcal{A}}
\newcommand{\cC}{\mathcal{C}}
\newcommand{\cE}{\mathcal{E}}
\newcommand{\cF}{\mathcal{F}}
\newcommand{\cN}{\mathcal{N}}
\newcommand{\cP}{\mathcal{P}}
\newcommand{\cQ}{\mathcal{Q}}
\newcommand{\cR}{\mathcal{R}}
\newcommand{\cv}{\alpha}
\newcommand{\cvar}{z}
\newcommand{\dashname}[1]{\stackrel{#1}{\begin{picture}(22,3)\put(0,2.5){\line(1,0){22}}\end{picture}}}
\newcommand{\dol}[1]{\overline{\overline{#1}}}
\newcommand{\grep}{\gv}
\newcommand{\Gr}{\mathrm{Gr}}
\newcommand{\gv}{\omega}
\newcommand{\into}{\hookrightarrow}
\newcommand{\kk}{\Bbbk}
\newcommand{\KQ}{K_0(Q)}
\newcommand{\loopvar}{z}
\newcommand{\NN}{\mathbb{N}}
\newcommand{\ol}[1]{\overline{#1}}
\newcommand{\onto}{\twoheadrightarrow}
\newcommand{\Qdp}{\widetilde{Q}_{dp}}
\newcommand{\Qrep}{M}
\newcommand{\rep}{\operatorname{rep}}
\newcommand{\TT}{\mathbb{T}}
\newcommand{\Zidx}{\ell}
\newcommand{\ZZ}{\mathbb{Z}}
\renewcommand{\mod}[1]{\langle {#1} \rangle}

\newtheorem{theorem}{Theorem}[section]
\newtheorem{conjecture}[theorem]{Conjecture}

\newtheorem{definition}[theorem]{Definition}
\newtheorem{lemma}[theorem]{Lemma}
\newtheorem{proposition}[theorem]{Proposition}
\theoremstyle{remark}
\newtheorem{example}[theorem]{Example}
\newtheorem{remark}[theorem]{Remark}
\numberwithin{equation}{section}
\numberwithin{figure}{section}

\begin{document}
\title{On Generalized Minors and Quiver Representations}

\author[Dylan Rupel]{Dylan Rupel}
\address[Dylan Rupel]{University of Notre Dame}
\email{drupel@nd.edu}

\author[Salvatore Stella]{Salvatore Stella}
\address[Salvatore Stella]{Universit\`a degli studi di Roma ``La Sapienza''}
\email{stella@mat.uniroma1.it}

\author[Harold Williams]{Harold Williams}
\address[Harold Williams]{University of Texas at Austin}
\email{hwilliams@math.utexas.edu}

\begin{abstract}
  The cluster algebra of any acyclic quiver can be realized as the coordinate ring of a subvariety of a Kac-Moody group -- the quiver is an orientation of its Dynkin diagram, defining a Coxeter element and thereby a double Bruhat cell.
  We use this realization to connect representations of the quiver with those of the group.
  We show that cluster variables of preprojective (resp. postinjective) quiver representations are realized by generalized minors of highest-weight (resp. lowest-weight) group representations, generalizing results of Yang-Zelevinsky in finite type.
  In type $A_n^{\!(1)}$ and finitely many other affine types, we show that cluster variables of regular quiver representations are realized by generalized minors of group representations that are neither highest- nor lowest-weight; we conjecture this holds more generally.
\end{abstract}

\maketitle

\tableofcontents
\newpage

\section{Introduction}

Given a finite quiver $Q$ without oriented cycles, there is an associated group $G$: the Kac-Moody group whose Dynkin diagram is the underlying unoriented graph of $Q$.
The representation theory of $Q$ is connected to $G$ in a range of ways; for example, one has the classical result of \cite{BGP73} that if $Q$ is of ADE type, its indecomposable representations are classified by the positive roots of $G$.
The connection which motivates the present work is grounded in the following trichotomies in the representation theories of $G$ and $Q$.

Recall that the representation theory of $G$ is largely centered around its (dual) categories of highest- and lowest-weight representations \cite{Kac90,Kum02}.
These coincide exactly when $G$ is a semisimple algebraic group, in which case they include all irreducible weight representations of $G$. 
In general, $G$ has many other such representations \cite{Kas94}, which are only well-understood when $G$ is of affine type \cite{Cha86,CP88}.
In the affine case, an irreducible representation is highest-weight, lowest-weight, or neither if its level (the character by which $Z(G) \cong \kk^\times$ acts) is positive, negative, or zero, respectively.

On the other hand, the representation theory of $Q$ is organized by the Auslander-Reiten translation $\tau: \rep Q \to \rep Q$ \cite{ASS06}. 
A representation $M$ is said to be preprojective if $\tau^k(M)$ is projective for some $k\ge 0$, postinjective if it is of the form $\tau^k(I)$ for some injective $I$ and some $k\ge0$, and regular otherwise.
Preprojective and postinjective representations coincide exactly when $Q$ is of ADE type, in which case there are no regular representations. When $Q$ does have regular representations, the tame-wild dichotomy implies that they are essentially unclassifiable unless $Q$ is of affine type, in which case they can be explicitly described.

Our goal here is to propose and partially demonstrate a way in which the parallel nature of these classifications reflects finer structural relationships between individual representations of $Q$ and $G$. Specifically, we identify certain generating functions attached to representations of $Q$ with a priori unrelated functions attached to representations of $G$, in such a way that the resulting association between representations of $Q$ and $G$ intertwines the above classifications.

This relationship is naturally couched in the language of cluster algebras, commutative rings equipped with a (partial) canonical basis whose elements are called cluster variables \cite{FZ02}.
From the perspective of $G$, the orientation $Q$ of its Dynkin diagram is equivalent to a choice of Coxeter element $c$ in its Weyl group; the cluster algebra $\cA_Q$ (with suitable coefficients) is the coordinate ring of the double Bruhat cell
\[
  G^{c,c^{-1}} := B_+ c B_+ \cap B_- c^{-1} B_- \subset G,
\]
where $B_+$ and $B_-$ are opposite Borel subgroups \cite{BFZ05,Wil13}.
The varieties $G^{c,c^{-1}}$ generalize the space of tridiagonal matrices with determinant equal to one, which they recover when $G=SL_n$ and $c$ is the standard Coxeter element.

From the perspective of $Q$, $\cA_Q$ is a repository for information about the submodule structure of representations of $Q$.
After embedding $\cA_Q$ into the ring of Laurent polynomials in the initial cluster variables, the remaining cluster variables are in bijection with the rigid indecomposable representations of $Q$ \cite{CC06,CK06}.
This correspondence takes a representation $\Qrep$ to a generating function of the Euler characteristics of its varieties of submodules; the resulting expression is called the cluster character of $\Qrep$ \cite{Pal12}.
With this in mind we refer to a (non-initial) cluster variable as preprojective, postinjective, or regular according to the classification of its associated representation.
In summary, given the description of $\cA_Q$ in terms of $G$, representations of $Q$ naturally label functions on the subvariety $G^{c,c^{-1}}$ of $G$.

When $Q$ is of ADE type and $G$ a simple algebraic group, Yang and Zelevinsky \cite{YZ08} showed that these functions have an interpretation in terms of the representation theory of $G$: they are restrictions of generalized minors.
These are functions $\Delta_{(V,\lambda)} \in \kk[G]$ labelled by a representation $V$ and a weight $\lambda$ for which the weight space $V_\lambda$ is one-dimensional.
The value of $\Delta_{(V,\lambda)}$ on $g \in G$ is the ratio of any nonzero $v \in V_\lambda$ with the projection of $gv \in V$ to $V_\lambda$.

Our main result is that cluster variables are also realized as generalized minors on $G^{c,c^{-1}}$ when $Q$ is not necessarily of finite type, and that moreover the resulting correspondence between quiver representations and group representations intertwines the triadic classifications described above.

\begin{theorem}
  \label{thm:maintheorem}
  Let $Q$ be a quiver without oriented cycles, $\cA_{\Qdp}$ the associated cluster algebra with doubled principal coefficients, and $G^{c,c^{-1}} \subset G$ the associated Coxeter double Bruhat cell.
  \begin{enumerate}
    \item
      The preprojective cluster variables in $\cA_{\Qdp} \cong \kk[G^{c,c^{-1}}]$ are restrictions of generalized minors of highest-weight representations.
      The postinjective cluster variables are restrictions of generalized minors of lowest-weight representations.
    \item
      Suppose $Q$ is an acyclic orientation of an $n$-cycle, so that $G$ is the universal central extension $\widetilde{LSL}_n$ of the loop group of $SL_n$.
      Then the regular cluster variables in $\cA_{\Qdp}\cong \kk[\widetilde{LSL}_n^{c,c^{-1}}]$ are restrictions of generalized minors of level zero representations.
  \end{enumerate}
Moreover, in all cases the natural identification of the $\bfg$-vector lattice of $Q$ with the weight lattice of $G$ takes the $\bfg$-vector of a cluster variable to the weight of the corresponding minor.
\end{theorem}

Though above we have taken $G$ to be of symmetric type for simplicity, part (1) also holds in the symmetrizable case upon generalizing from representations of quivers to representations of valued quivers or species.
We emphasize that although the coordinate ring of any double Bruhat cell $G^{u,v}$ is an (upper) cluster algebra, its defining quiver is generally not an orientation of the Dynkin diagram of $G$.
More to the point, in this generality one expects that most cluster variables do not admit an elementary Lie-theoretic description, e.g. as restrictions of minors.

The $\bfg$-vector of a cluster variable \cite{FZ07} is an element of $\ZZ^{|Q_0|}$, referred to as the $\bfg$-vector lattice in this context, which uniquely labels it among all cluster variables. 
In terms of the quiver representation associated to a cluster variable, the $\bfg$-vector records the modules appearing in a minimal injective copresentation.
For acyclic $Q$, we identify $\ZZ^{|Q_0|}$ with the weight lattice of $G$ by taking the initial $\bfg$-vectors to the fundamental weights.
The last part of the Theorem then asserts that this identifies the $\bfg$-vector of any cluster variable with the weight of the associated minor (in finite type this is a result of \cite{YZ08}).

While the proof of (1) is mostly a straightforward generalization of the finite-type case, the proof of (2) demands different methods.
The reason (1) is true is that the preprojective and postinjective cluster variables are determined by an explicit list of relations, and these are implied by certain generalized determinantal identities that hold globally on $G$.
However, these identities depend crucially on the fact that all minors involved are semi-invariants of the standard Borel subgroups and their conjugates.
The minors arising from level zero representations do not have this property and so the relations that determine them cannot be deduced from such identities.

Instead, we study level zero minors on $\widetilde{LSL}_n$ by direct combinatorial means.
We may take the relevant representations to be of the form $\big(\!\bigwedge^{\!k}\kk^n\big)[\loopvar^{\pm 1}]$, hence they inherit natural bases from the standard basis of $\kk^n$.
These bases let us compute the restrictions to $G^{c,c^{-1}}$ of the level zero minors appearing in (2) as weighted sums over collections of pairwise disjoint paths through a directed network on the cylinder, following similar constructions in \cite{GSV12,FM14}.
When $Q$ is an acyclic orientation of an $n$-cycle, its regular representations likewise have a simple combinatorial description, and one can construct by hand a bijection between their subrepresentations and collections of paths in the relevant network.

A deficiency of proving (2) by analyzing each regular cluster variable directly is that we do not know how to interpret the exchange relations relating them to the other cluster variables in terms of the representation theory of $\widetilde{LSL}_n$.
It would be desirable to find such an interpretation, which could also provide an avenue for approaching more general Kac-Moody groups:

\begin{conjecture}
\label{conj:mainconjecture}
  Let $Q$ be a finite (valued) quiver without oriented cycles and $G^{c,c^{-1}} \subset G$ the associated Coxeter double Bruhat cell.
  Then all regular cluster variables in  $\cA_{\Qdp} \cong \kk[G^{c,c^{-1}}]$ are generalized minors of representations that are neither highest- nor lowest-weight, and the weight of this minor is the $\bfg$-vector of the cluster variable.
\end{conjecture}

In addition to verifying the conjecture for all Coxeter elements when $G$ is of type $A_n^{(1)}$, we check a finite list of other affine types, see Section~\ref{sec:othertypes}.
In any affine type, the regular cluster variables are finite in number and can be explicitly computed. 
In work of the second author it is shown that their $\bfg$-vectors correspond to weights in the complement of the Tits cone and its negative \cite{RS16}.
In particular, it follows that if these cluster variables are restrictions of minors, they must come from explicitly identifiable level zero representations.
The types checked in Section~\ref{sec:othertypes} are essentially the ones where the relevant minors on $G^{c,c^{-1}}$ can be computed directly from the characters of these representations and compared with the corresponding cluster variables. 
Beyond finite and affine types, general weight representations of $G$ are not well-understood and the conjecture is much more speculative (though we verify a related result for all rank 2 Kac-Moody groups in Theorem~\ref{thm:homogeneous}).

Finally, in type $A_{n}^{\!(1)}$ we also prove the following extension of Theorem~\ref{thm:maintheorem}.
There is a one-parameter family of quiver representations whose dimension vector is the primitive imaginary root.
The cluster character of a generic element of this family is not a cluster variable, but is an element of the generic basis \cite{Dup12} (and we expect of other standard canonical bases as well \cite{BZ14,GHKK14}).
We show that this element too is a level zero generalized minor; with this in mind it is natural to further ask which canonical basis elements other than cluster variables can be realized as generalized minors.

\textsc{Acknowledgments}
We would like to thank Giovanni Cerulli Irelli, Vyjayanthi Chari, Alessandro D'Andrea, Corrado de Concini, Anna Felikson, Michael Gekhtman, Bernard Leclerc, Nathan Reading and Pavel Tumarkin for helpful discussions and comments.
D.R. is partially supported by an AMS-Simons Travel Grant.
S.S. is a Marie Curie fellow of the Istituto Nazionale di Alta Matematica.
H.W. is supported by an NSF Postdoctoral Research Fellowship DMS-1502845.

\section{Background on minors, clusters, and quivers}

We collect here some needed material on cluster algebras and representation theory, and lay out the notation we will use.
We fix throughout an algebraically closed field $\kk$ of characteristic zero.

\subsection{Kac-Moody groups and generalized minors}
\label{sec:group background}

For any symmetrizable Cartan matrix $A$ there is an associated Kac-Moody group $\widehat{G}:=\widehat{G}_A$.
It has several variants, and we consider the minimal version studied in \cite{KP83} and \cite[Section 7.4]{Kum02}.
This is an ind-algebraic group whose derived subgroup $G$ is generated by coroot subgroups $\varphi_i: SL_2 \into G$ for $1 \leq i \leq n$.
Starting from $G$, the larger group $\widehat{G}$ is obtained as a semidirect product with an algebraic torus of dimension equal to the corank of $A$.

When $A$ is of finite type, $G \cong \widehat{G}$ is a simply-connected semisimple algebraic group.
When $A$ is of (untwisted) affine type, $G$ can be identified with the universal central extension of the group $LG^\circ$ of algebraic loops (that is, regular maps from $\kk^\times$) into some simply-connected semisimple algebraic group $G^\circ$; in this case $\widehat{G}$ is obtained by adjoining a factor of $\kk^\times$ acting by loop rotations.

For $t \in \kk^\times$ and $1 \leq i \leq n$ we adopt the notation
\begin{gather}
  x_{i}(t):=\varphi_i\begin{pmatrix} 1 & t \\ 0 & 1\end{pmatrix}
  \quad
  t^{\alpha_i^\vee}:=\varphi_i\begin{pmatrix}t & 0 \\ 0 & t^{-1}\end{pmatrix}
  \quad
  x_{\ol{\imath}}(t):=\varphi_i\begin{pmatrix} 1 & 0 \\ t & 1\end{pmatrix}
  \\
  \ol{s_{i}} := \varphi_i \begin{pmatrix} 0 & -1 \\ 1 & 0 \end{pmatrix}
  \quad
  \dol{s_{i}} := \varphi_i \begin{pmatrix} 0 & 1 \\ -1 & 0 \end{pmatrix}.
\end{gather}
By extension, for any $w$ in the Weyl group $W$ of $\widehat{G}$ we set \[\ol{w} := \ol{s_{i_1}}\cdots\ol{s_{i_k}}, \quad \quad \dol{w} := \dol{s_{i_1}}\cdots\dol{s_{i_k}},\] where $s_{i_1}\cdots s_{i_k}$ is any reduced word for $w$.
We write $H$ for the ($n$-dimensional) Cartan subgroup of $G$, likewise $\widehat{H}$ is the Cartan subgroup of $\widehat{G}$ (it is isomorphic to $H \times (\kk^\times)^{\corank(A)}$).
We let $\kk[\widehat{G}]$ denote the (complete topological) algebra of regular functions on $\widehat{G}$.

We have weight lattices $\widehat{P} = \Hom(\widehat{H},\kk^\times)$ and $P = \Hom(H,\kk^\times)$, and write the value of $\lambda \in P$ on $h \in H$ as $h^\lambda$ (likewise for $\lambda \in \widehat{P}$ and $h \in \widehat{H}$).
The fundamental weights $\omega_1,\dotsc,\omega_n \in P$ are defined by $(t^{\alpha_i^\vee})^{\omega_j} = t^{\delta_{ij}}$.
The inclusion $H \into \widehat{H}$ induces a projection $\widehat{P} \onto P$, and we fix once and for all a splitting $P \into \widehat{P}$ (this is equivalent to fixing an isomorphism $\widehat{G} \cong G \rtimes (\kk^\times)^{\corank(A)}$).
Note that while $\widehat{P} \onto P$ is $W$-equivariant, the splitting will generally not be.

Generalized minors are certain functions on $\widehat{G}$, recovering the usual notion of minors when $\widehat{G} \cong SL_n$.
The general definition involves the representation theory of $\widehat{G}$.
By a \newword{weight representation} $V$ of $\widehat{G}$ we will always mean an algebraic representation that splits as a direct sum
\[
  V = \bigoplus_{\lambda \in \widehat{P}} V_\lambda
\]
of finite-dimensional weight spaces under the action of $\widehat{H}$.
Note that the weight spaces of $V$ as a $G$-representation are not necessarily finite-dimensional, as the weight decomposition of the $\widehat{H}$-action is finer than that of the $H$-action.

A weight $\lambda \in \widehat{P}$ is \newword{dominant} if $(t^{\alpha_i^\vee})^\lambda$ is a nonnegative power of $t$ for all $1 \leq i \leq n$, and we denote the set of dominant weights by $\widehat{P}^+$.
To each $\lambda \in \widehat{P}^+$ is associated an irreducible highest-weight representation $V(\lambda)$ of $\widehat{G}$.
Similarly, for each $\lambda \in -\widehat{P}^+$ there is an irreducible lowest-weight representation $V(\lambda)$ of $\widehat{G}$.
For any $w\in W$ and $\lambda\in\pm\widehat{P}^+$, the weight space $V(\lambda)_{w\lambda}$ is one-dimensional and thus we also write $V(w\lambda)$ for the representation $V(\lambda)$.
The \newword{Tits cone} $$X = \{w\lambda : w \in W, \lambda \in \widehat{P}^+\}$$ is the set of weights which can be obtained from a dominant weight by the action of the Weyl group.
Thus each weight $\lambda\in X \cup -X$ determines an irreducible (highest- or lowest-) weight representation $V(\lambda)$ of $\widehat{G}$ whose weight space $V(\lambda)_\lambda$ is one-dimensional.

More generally, when $\lambda$ is not necessarily in $X \cup -X$ we can consider representations $V$ of $\widehat{G}$ for which $\lambda$ is extremal; that is, $V_\lambda$ and all of its Weyl group conjugates are highest- or lowest-weight spaces for each coroot subgroup $\varphi_i(SL_2)$. For arbitrary $\lambda$ an example of such a representation is contructed in \cite{Kas94}, but only in affine types are a classification and explicit realization available \cite{Cha86,CP88}.

Specifically, when $A$ is of (untwisted) affine type an identification of $G$ with a central extension of $LG^\circ$ fixes a splitting $P \cong P^\circ \oplus \ZZ \kappa$, where $P^\circ$ is the weight lattice of the underlying semisimple group $G^\circ$ and $\kk \kappa = \Hom( Z(G),\kk^\times)$.
Above this we have the splitting $\widehat{P} \cong P \oplus \ZZ \delta$, where $\delta$ is a character of the $\kk^\times$ acting by loop rotations.
A weight $\lambda \in \widehat{P}$ is in the complement of $X \cup -X$ if and only if it lies in $P^\circ \oplus \ZZ \delta$.
Following the above notation, for $\lambda^\circ\in P^\circ$ we have an irreducible $G^\circ$-representation $V(\lambda^\circ)$ whose highest weight is conjugate to $\lambda^\circ$ under the Weyl group of $G^\circ$.
With this in mind when $G$ is affine and $\lambda=\lambda^\circ+r\delta\in \widehat{P} \smallsetminus(X \cup -X)$ we set $V(\lambda) := V(\lambda^\circ)[\loopvar^{\pm 1}]$, the space of algebraic loops in $V(\lambda^\circ)$.

The \newword{level} of an irreducible $\widehat{G}$-representation is the character by which $Z(G) = Z(\widehat{G}) \cong \kk^\times$ acts; the level of $V(\lambda)$ is positive if $\lambda \in X$, negative if $\lambda \in -X$, and zero if $\lambda \in \widehat{P} \smallsetminus(X \cup -X)$.

\begin{definition}
  \label{def:minors}
  Let $V$ be a weight representation of $\widehat{G}$ and $\lambda \in \widehat{P}$ a weight for which the weight space $V_\lambda$ is one-dimensional.
  The principal generalized minor $\Delta_{(V,\lambda)} \in \kk[\widehat{G}]$ is the function
  \[
    g \mapsto \pi_\lambda(gv_\lambda)/v_\lambda,
  \]
  where $\pi_\lambda:V \onto V_\lambda$ is the orthogonal projection and $v_\lambda \in V_\lambda$ is any nonzero vector.

  When $\lambda \in X \cup -X$, or when $\widehat{G}$ is of untwisted affine type and $\lambda$ is arbitrary, we simply write $\Delta_\lambda$ for $\Delta_{(V(\lambda),\lambda)}$.
  When $\lambda \in \widehat{P}^+$ and $u,v\in W$, the generalized minor $\Delta^{u\lambda}_{v\lambda}$ is the function $g \mapsto \Delta_\lambda(\ol{u}^{-1}g\ol{v})$.
  If $\lambda \in -\widehat{P}^+$ and $u,v\in W$, then $\Delta^{u\lambda}_{v\lambda}$ is the function $g \mapsto \Delta_\lambda(\dol{u}^{-1}g\dol{v})$.
\end{definition}

In other words, if we use a weight basis of $V$ to write the action of $g$ as a matrix, $\Delta_{(V,\lambda)}(g)$ is the diagonal entry corresponding to $V_\lambda$.
Similarly, $\Delta^{u\lambda}_{v\lambda}(g)$ is the matrix entry whose row corresponds to $V_{v\lambda}$ and whose column corresponds to $V_{u\lambda}$ (up to an overall scalar fixed by the chosen representatives of $u$, $v \in G$).
While we have apparently given two definitions of $\Delta^{u\lambda}_{v\lambda}$ in finite type (depending on whether we regard $u\lambda$ and $v\lambda$ as conjugate to a dominant or anti-dominant weight), the reader may check that these definitions agree in this case.

\begin{remark}
  When $\widehat{G}$ is of (untwisted) affine type, the representations $V(\lambda)$ with $\lambda \in \widehat{P}\smallsetminus(X \cup - X)$ are not the only irreducible level zero weight representations.
  The set of all such representations is parametrized by pairs $(\{\lambda_1,\dotsc,\lambda_m\},\{a_1,\dotsc,a_m\})$ given by a tuple of dominant weights of $G^\circ$ and a tuple of elements of $\kk^\times$ \cite{CP86}.
  As a vector space the associated representation is the space of Laurent polynomials valued in the tensor product of the $V(\lambda_i^\circ)$, and the scalars $a_i$ rescale the action of $G^\circ$-valued loops on each factor.
\end{remark}

A crucial feature of generalized minors is the following generalized determinantal identity:

\begin{proposition}\cite{FZ99,Wil13}
  \label{prop:fundid}
  Suppose $u,v \in W$ and $1 \leq i \leq n$ are such that $\ell(u)<\ell(us_i)$ and $\ell(v)<\ell(vs_i)$.
  Then
  \begin{equation}
    \label{eq:fundid}
    \Delta_{v\omega_i}^{u\omega_i} \Delta_{vs_i\omega_i}^{us_i\omega_i}
    =
    \prod_{\stackrel{1\leq j \leq n}{j\neq i}}\left(\Delta_{v\omega_j}^{u\omega_j}\right)^{-a_{ji}}
    +
    \Delta_{vs_i\omega_i}^{u\omega_i} \Delta_{v\omega_i}^{us_i\omega_i}
  \end{equation}
  is satisfied on $G$.
\end{proposition}

While the above identity on $G$ is the one we use directly, we note that it restricts from a similar one on $\widehat{G}$ that also involves the characters of $\widehat{G}/G$.

\subsection{Cluster algebras and $\bfg$-vectors}
\label{sec:clusteralgebras}
An $n\times n$ matrix $B$ is said to be \newword{skew-symmetrizable} if it can be transformed to a skew-symmetric matrix through multiplication by a diagonal matrix.
For $m \geq n$, the \newword{principal part} of an $m \times n$ matrix is the submatrix formed by its first $n$ rows.
An \newword{exchange matrix} $\widetilde B$ is an $m\times n$ integer matrix whose principal part $B$ is skew-symmetrizable.
We will refer to such a $\widetilde B$ as an \newword{extension} of $B$.

Two $m\times n$ exchange matrices $\widetilde B=(b_{ij})$ and $\widetilde B'=(b'_{ij})$ are related by \newword{matrix mutation} at $k \in \{1,\dotsc,n\}$ if their entries satisfy
\begin{equation}
  \label{eq:matrix mutation}
  b'_{ij} = \begin{cases}
  -b_{ij} & i = k \text{ or } j = k\\
  b_{ij} + [b_{ik}]_+ [b_{kj}]_+ - [-b_{ik}]_+ [-b_{kj}]_+ & \text{otherwise.}
  \end{cases}
\end{equation}
Here and elsewhere we write $[a]_+$ for $\max(a,0)$.
Note that mutation is an involution and defines an equivalence relation on exchange matrices.

To define the cluster algebra $\cA_{\widetilde B}$ associated to an $m \times n$ exchange matrix $\widetilde B$, we begin with an infinite $n$-ary tree $\TT$ with edges labelled by $\{1,\dotsc,n\}$ so that the edges incident to any given vertex have distinct labels.
Fix a root vertex $t_0 \in\TT$.
Assign an exchange matrix $\widetilde B^t=(b_{ij}^t)$ to each $t \in\TT$ so that:
\begin{itemize}
  \item 
    $\widetilde B^{t_0} = \widetilde B$,
  \item 
    if $t, t' \in\TT$ are joined by an edge labelled $k$, then $\widetilde B^t$ and $\widetilde B^{t'}$ are related by mutation at $k$.
\end{itemize}

Let $\cF$ denote the field of rational functions in formal variables $x_{1}\,\dotsc,x_{m}$ with coefficients in $\kk$.
The \newword{cluster variables}
\[
  \big\{ x_{i;t}: i \in [1,m], t \in\TT\big\} \subset \cF
\]
are defined recursively as follows.
The \newword{initial cluster variables} $x_{i;t_0}$ are taken to be the generators $x_i$.
If $t,t' \in \TT$ are joined by an edge labelled $k$, then $x_{i;t} = x_{i;t'}$ for $i \neq k$ and $x_{k;t}$, $x_{k;t'}$ are related by the \newword{exchange relation}
\[
  x_{k;t}x_{k;t'}
  =
  \prod_{b^t_{ik}>0}x_{i;t}^{b^t_{ik}}
  +
  \prod_{b^t_{ik}<0}x_{i;t}^{-b^t_{ik}}.
\]
The variables $x_{n+1;t},\dotsc, x_{m;t}$ do not depend on $t$ and hence are referred to as \newword{frozen variables}.
We also refer to variables $x_{i;t}$ for $i\in[1,n]$ as \newword{mutable}.

\begin{definition}[\cite{FZ02}]
  The (geometric type) \newword{cluster algebra} $\cA_{\widetilde B}$ is the $\kk$-subalgebra of $\cF$ generated by the set of all cluster variables.
\end{definition}

We say a skew-symmetrizable matrix is \newword{acyclic} if it can be conjugated by a permutation matrix so that the entries above the diagonal all become nonnegative.
In other words, $B$ is acyclic if there is a permutation $\sigma \in S_n$ such that $b_{\sigma_i \sigma_j} \geq 0$ when $i < j$.
This terminology is justified by the skew-symmetric case: the data of a skew-symmetric matrix $B$ is the same as that of a quiver $Q := Q_B$ with vertices indexed by the columns of $B$ and $[-b_{ij}]_+$ arrows from vertex $i$ to vertex $j$; the skew-symmetric matrix is acyclic exactly when this quiver has no oriented cycles.
We call an exchange matrix acyclic if its principal part is acyclic.

\begin{remark}
  Just as a skew-symmetric matrix $B$ can be encoded as a quiver $Q$, an extension $\widetilde{B}$ of $B$ can be encoded as a \newword{framed quiver} $\widetilde{Q}$.
  This means a quiver obtained from $Q$ by adjoining $m -  n$ \newword{frozen vertices}.
  These have no arrows between them and the arrows connecting them to the original vertices are determined in the evident way by the non-principal part of $\widetilde{B}$.
  Given a framed quiver $\widetilde{Q}$, we simply write $\cA_{\widetilde{Q}}$ for the cluster algebra associated to its signed adjacency matrix $\widetilde{B}$ as above.
\end{remark}

\begin{figure}
  \centering
  \begin{tikzpicture}
    \node (l) [matrix] at (9,0) {
      \coordinate (a) at (9,0);
      \coordinate (b) at (11,0);
      \coordinate (c) at (10,-1.3);
      \foreach \c in {a,b,c} {\fill (\c) circle (.06);}
      \foreach \c/\d in {a/b,c/b,c/a} {\draw[thick,-stealth',shorten <=2mm,shorten >=2mm] (\c) to (\d);}
      \node (alabel) at ($(a)+(-.3,.1)$) {1};
      \node (blabel) at ($(b)+(.3,.1)$) {2};
      \node (clabel) at ($(c)+(0,-.35)$) {3};
    \\
    };
    \node (Q) at (7.7,0) {$Q =$};
    \node (m) at (0,0) {$B = \begin{bmatrix} 0 & -1 & 1 \\ 1 & 0 & 1 \\ -1 & -1 & 0 \end{bmatrix}$};
    \node (r) at (4.5,0) {$A = \begin{bmatrix} 2 & -1 & -1 \\ -1 & 2 & -1 \\ -1 & -1 & 2 \end{bmatrix}$};
  \end{tikzpicture}
  \caption{
    An acyclic skew-symmetric matrix $B$, its Cartan companion $A$, and the corresponding quiver $Q$.
    The vertex labels of $Q$ correspond to the row/column labels of $B$ and $A$, and the Coxeter element associated to $B$ (equivalently, to $Q$) is $c = s_2 s_1 s_3$.
  }
  \label{fig:matrices}
\end{figure}

Yet another way of encoding an acyclic skew-symmetrizable matrix $B$ is a symmetrizable Cartan matrix $A$ (the \newword{Cartan companion} of $B$) together with a Coxeter element $c$ of its Weyl group.
The Cartan matrix is obtained from $B$ by making all its entries negative without changing their absolute values, then setting each diagonal entry equal to two; see Figure~\ref{fig:matrices}.
The Coxeter element is $c = s_{\sigma_1}\cdots s_{\sigma_n}$, where $\sigma$ is any permutation with $b_{\sigma_i \sigma_j} \geq 0$ when $i < j$ as above (choosing a different permutation with this property defines the same Coxeter element via a different reduced word).
Conversely, a Cartan matrix together with a Coxeter element $c$ determines a skew-symmetrizable matrix $B$ by designating signs for the off-diagonal entries according to any reduced word for $c$.

For acyclic cluster algebras we will be especially interested in the following sequence of vertices in $\TT$ and the cluster variables associated to them.
Letting $t_0 \in \TT$ denote a fixed root vertex as before, we fix a permutation $\sigma$ as above and write $\{t_\Zidx\}_{\Zidx \in \ZZ} \subset \TT$ for the vertices along the unique path with edges labelled as follows:
\begin{equation}
  \cdots
  \dashname{\sigma_n}
  t_{-n}
  \dashname{\sigma_1}
  \cdots
  \dashname{\sigma_{n-1}}
  t_{-1}
  \dashname{\sigma_n}
  t_0
  \dashname{\sigma_1}
  t_1
  \dashname{\sigma_2}
  \cdots
  \dashname{\sigma_n}
  t_n
  \dashname{\sigma_1}
  \cdots
  \label{eq:path}
\end{equation}
It is straightforward to check that this is a subpath in the \newword{acyclic belt} of $\TT$, the full subgraph whose set of vertices consists of those $t \in \TT$ for which $\widetilde{B}^t$ is acyclic.

While the vertex $t_\ell$ associated to $\Zidx \in \ZZ$ in general depends on the choice of $\sigma$, the set
\begin{equation}
  \label{eq:cluster_variables_on_path}
  \{x_{i;t_\Zidx}: i \in [1,n], \Zidx \in \ZZ\}
\end{equation}
of cluster variables appearing along the above path does not.
In fact, \eqref{eq:cluster_variables_on_path} is equal to the set of all cluster variables $x_{i;t}$ for which $\widetilde{B}^t$ is acyclic \cite[Corollary 4]{CK06} (while the proof is only written for the skew-symmetric case, the same argument can be applied in the skew-symmetrizable case).
We will refer to the non-initial cluster variables $x_{i;t_\Zidx}$ with $\Zidx>0$ (resp. $\Zidx<0$) as \newword{preprojective} (resp. \newword{postinjective}), and to a cluster variable that does not appear along the path \eqref{eq:path} as \newword{regular}; this terminology is justified by Proposition~\ref{prop:ccvsacyclicbelt}.
We note that when $n=2$ there are no regular cluster variables.

For each $t \in\TT$ and $1 \leq i \leq n$ there is an auxiliary pair of integer vectors depending only on the principal part $B$: the \newword{$\bfc$-vector} $\cv_{i;t} \in \ZZ^n$ and the \newword{$\bfg$-vector} $\gv_{i;t} \in \ZZ^n$ \cite{FZ07}.
We write $C^B_t$ and $G^B_t$ for the matrices whose $i$-th columns are $\cv_{i;t}$ and $\gv_{i;t}$, respectively.
To define these vectors, consider the exchange matrix
\[
  \Bpr := \begin{bmatrix} B \\ Id_n \end{bmatrix},
\]
which is said to have \newword{principal coefficients}.
For $t\in\TT$, if $\Bpr^t$ is the exchange matrix obtained by iterated mutation from $\Bpr^{t_0}=\Bpr$, then $C^B_t$ is defined to be the bottom $n \times n$ submatrix of $\Bpr^t$.

Given the matrix of $\bfc$-vectors at $t\in\TT$, we may define the matrix of $\bfg$-vectors at the same vertex $t$ by
\begin{equation}
  \label{eq:g-vectors_from_c-vectors}
  (G^B_t)^T = (C_t^{-B^T})^{-1},
\end{equation}
where $X^T$ denotes the transpose of a matrix $X$.
Note that this is not the standard definition, but the above characterization (due to \cite[Theorem 1.2]{NZ12} and the sign-coherence result of \cite{DWZ10,GHKK14}) is what we will use in Section~\ref{sec:acyclic_belt}.
The basic role of $\bfg$-vectors is that they provide natural labels of cluster variables: if $\widetilde{B}$ is of full rank, then two cluster variables $x_{i;t}$ and $x_{j;t'}$ coincide in $\cA_{\widetilde B}$ if and only if $\gv_{i;t}$ and $\gv_{j;t'}$ coincide in $\ZZ^n$ \cite{DWZ10,GHKK14}.
We will therefore often denote the cluster variable with $\bfg$-vector $\gv$ by $x_\gv$.

\begin{remark}
  \label{rem:cgrootsweights}
  As our notation is meant to suggest, when $B$ is acyclic and $A$ is its Cartan companion we consistently identify the $\bfg$-vector lattice with the weight lattice $P$ of $A$ by taking the initial $\bfg$-vector $\gv_{i;t_0}$ to the fundamental weight $\omega_i$.
  Similarly we identify the $\bfc$-vector lattice with the root lattice $\cQ$ by taking the initial $\bfc$-vector $\cv_{i;t_0}$ to the simple root $\alpha_i$.
  The duality between $\bfc$- and $\bfg$-vectors expressed in Equation~\eqref{eq:g-vectors_from_c-vectors} should be interpreted as the duality between weights and coroots (i.e. roots of the Langlands dual).
\end{remark}

\subsection{Quiver representations and cluster characters}
\label{sec:quiverbackground}
A \newword{quiver} $Q=(Q_0,Q_1,s,t)$ with $n$ vertices consists of a set $Q_0=\{1,2,\ldots,n\}$ of vertices, a set $Q_1$ of arrows, and maps $s,t:Q_1\to Q_0$ giving the source and target of an arrow.
A \newword{$\kk$-representation} $M=(M_i,M_a)$ of $Q$ consists of a $\kk$-vector space $M_i$ for each $i\in Q_0$ and a $\kk$-linear map $M_a:M_{s(a)}\to M_{t(a)}$ for each arrow $a\in Q_1$.
We denote by $\rep_\kk(Q)$ the hereditary abelian category of finite-dimensional representations, and refer to \cite{ASS06} for details on the material recalled below.

We will only be interested in the case when $Q$ is \newword{acyclic} (i.e. has no oriented cycles or self-loops), hence we assume this throughout.
This condition holds if and only if \newword{path algebra} $\kk Q$ (the $\kk$-vector space on the set of directed paths in $Q$) is finite-dimensional.
As a $Q$-representation $\kk Q$ is the direct sum of the indecomposable projective representations of $Q$, so in particular these are also finite-dimensional (as are the indecomposable injective representations).

Given a representation $M\in\rep_\kk(Q)$ and a projective presentation $0\to P_1\to P_0\to M\to 0$ (which may be chosen to terminate after two terms since $\rep_\kk(Q)$ is hereditary), the \newword{Auslander-Reiten translate} $\tau(M)\in\rep_\kk(Q)$ is defined via the short exact sequence
\begin{equation}
  \label{eq:AR translation}
  0\longrightarrow \tau(M)\longrightarrow D\Hom_Q(P_1,\kk Q)\longrightarrow D\Hom_Q(P_0,\kk Q)\longrightarrow 0,
\end{equation}
where $D=\Hom(-,\kk)$ is the standard $\kk$-linear duality functor.

The \newword{AR translation} $\tau$ plays a key role in the representation theory of $Q$: a representation $P$ is projective if and only if $\tau(P)=0$ while a representation $I$ is injective if and only if it is not of the form $\tau(M)$ for any $M$.
Call a representation $M$ \newword{preprojective} if $\tau^k(M)$ is projective for some $k\ge0$, \newword{postinjective} if there is an injective representation $I$ so that $M=\tau^k(I)$ for some $k\ge0$, and \newword{regular} otherwise.
When $Q$ is an orientation of a simply-laced finite-type Dynkin diagram, the injective representations are preprojective (hence projective representations are postinjective) and there are no regular representations.

We write $\KQ$ for the \newword{Grothendieck group} of $\rep_\kk(Q)$; it is the abelian group freely generated by the classes $[S_i]$ of the vertex-simple representations of $Q$.
For $M\in\rep_\kk(Q)$ we call $[M]\in\KQ$ its \newword{dimension vector}, as $[\Qrep] = \sum_i \dim(M_i)[S_i]$.
Since $\rep_\kk(Q)$ is hereditary, the \newword{Euler-Ringel form} \[\langle M,N\rangle:=\dim_\kk\Hom(M,N)-\dim_\kk\Ext^1(M,N)\] only depends on the dimension vectors of $M$ and $N$, hence defines a bilinear form on $\KQ$.
In particular, it is completely determined by the pairings of vertex-simple representations:
\[
  \langle S_i,S_j\rangle=
  \begin{cases}
    1 & \text{if $i=j$;}\\
    -\#\{\text{arrows $i \to j$ in $Q_1$}\} & \text{if $i \neq j$.}
  \end{cases}
\]

If $\widetilde{Q}$ is a framed quiver of an acyclic quiver $Q$, the non-initial cluster variables in $\cA_{\widetilde{Q}}$ are in bijection with the rigid indecomposable representations of $Q$.
Recall that a representation $M$ of $Q$ is \newword{rigid} if $\Ext^1(M,M) = 0$.
The bijection is given by associating to a representation $M$ a generating function encoding its submodule structure, referred to as its \newword{cluster character} or \newword{Caldero-Chapoton function}.

The definition of the cluster character involves the \newword{$\bfg$-vector} $\grep(\Qrep)\in\ZZ^n$ of a representation $\Qrep$.
If
\[
  0 \to \Qrep \to \bigoplus I_j^{a_j} \to \bigoplus I_j^{b_j} \to 0
\]
is an injective resolution of $M$, then $\grep(\Qrep)$ is the element of $\ZZ^n$ whose $j$-th component is $b_j - a_j$.
The negative of $\grep(\Qrep)$ is often called the coindex of $M$, but the present terminology is justified by Theorem~\ref{thm:CCbijection}.
Below we write $x^{\grep(\Qrep)}$ for the monomial $\prod x_j^{b_j - a_j}$.

\begin{remark}
  As $Q$ is acyclic, it has an associated symmetric Cartan companion $A$.
  As in Section~\ref{sec:clusteralgebras} we identify the $\bfg$-vector lattice $\ZZ^n$ with the weight lattice $P$ associated to $A$ by mapping the $j$-th standard basis vector to the fundamental weight $\omega_j$.
  Thus, for example, $\grep(\Qrep) = \sum (b_j - a_j)\omega_j$ in the above definition.
\end{remark}

Given a dimension vector $e = \sum_i e_i[S_i] \in \NN^n \subset \KQ$, the \newword{quiver Grassmannian} $\Gr_eM$ is the projective variety of $e$-dimensional submodules of $M$.
We write $\chi(\Gr_eM)$ for its Euler characteristic (formally, in \'etale cohomology with compact supports).
If $\widetilde{B}=(b_{ij})$ is the exchange matrix associated to $\widetilde{Q}$ we write $\hat{y}_j$ for the monomial $\prod\limits_{i=1}^m x_i^{b_{ij}}$, and $\hat{y}^e$ for $\prod\limits_{j=1}^n \hat{y}_j^{e_j}$.

\begin{definition}
  Let $\widetilde{Q}$ be a framed quiver and $M$ a representation of $Q$.
  The (framed) cluster character of $M$ is
  \begin{equation}
    \label{eq:cc formula}
    x_M := x^{\grep(M)} \sum_{e \in \NN^n}\chi(\Gr_e M) \hat{y}^e \in \kk[x_1^{\pm 1},\dotsc,x_m^{\pm 1}].
  \end{equation}
\end{definition}

\begin{theorem}
  \label{thm:CCbijection}
  \cite{CC06,CK06}
  Suppose that $\widetilde{Q}$ has no arrows whose source is a frozen vertex (equivalently, the non-principal part of $\widetilde{B}$ has only positive entries).
  Then the assignment $M \mapsto x_M$ defines a bijection between the set of rigid indecomposable representations of $Q$ and the set of non-initial cluster variables in $\cA_{\widetilde{Q}}$.
  Moreover, the $\bfg$-vector of the cluster variable $x_M$ is equal to $\grep(M)$.
\end{theorem}

The positivity assumption on the non-principal part of $\widetilde{B}$ is only to keep the formulas from being more complicated than we need.
More generally, cluster variables and cluster characters will be related by a monomial determined by the separation of additions formula of \cite{FZ07}.

\begin{remark}
  Theorem~\ref{thm:CCbijection} is a simple instance of a number of closely related connections between cluster theory and the representation theory of algebras.
  Other connections involve, for example, preprojective algebras \cite{GLS13}, quivers with potential \cite{DWZ10}, and the theory of cluster categories \cite{BMRRT06}.
\end{remark}

With Theorem~\ref{thm:CCbijection} in hand we refer to non-initial cluster variables of $\cA_{\widetilde{Q}}$ as preprojective, postinjective, or regular if they come from a quiver representation of the corresponding type.
This is compatible with the discussion of the acyclic belt in Section~\ref{sec:clusteralgebras}:

\begin{proposition}
  \label{prop:ccvsacyclicbelt}
  A rigid indecomposable representation $\Qrep$ is preprojective (resp. postinjective) if and only if the cluster variable $x_M$ is of the form $x_{i;t_\Zidx}$ for some $\Zidx > 0$ (resp. $\Zidx < 0$) in the notation of Equation~\eqref{eq:cluster_variables_on_path}.
\end{proposition}
\begin{proof}
  Recall that the mutation sequence in Equation~\eqref{eq:path} defining the node $t_\Zidx$ for $\Zidx>0$ (resp. $\Zidx<0$) proceeds by mutating sinks (resp. sources) of $Q$.
  By \cite{DWZ10,Ru11}, the first mutation out of a given node always produces the cluster character of a vertex simple representation and mutations of the initial cluster at a sink/source acts on cluster characters via reflection functors acting on the representation.
  Using this and the interpretation of the Auslander-Reiten translation in terms of reflection functors \cite{BB76}, it is easy to see that the variable $x_{i;t_n}$ (resp. $x_{i;t_{-n}}$) is given by the cluster character $x_{P_i}$ (resp. $x_{I_i}$) of the projective cover (resp. injective hull) of the vertex simple $S_i$.
  Moreover, the same reasoning shows that the representations corresponding to variables $x_{\langle\Zidx\rangle;t_{\Zidx}}$ for $\Zidx>0$ (resp. $\Zidx<0$) sweep out all representations in the preprojective (resp. postinjective) component of the Auslander-Reiten quiver of $Q$.
\end{proof}

Theorem~\ref{thm:CCbijection} and Proposition~\ref{prop:ccvsacyclicbelt} generalize to the case of skew-symmetrizable matrices by replacing quiver representations with representations of valued quivers or species.
We refer to \cite{Ru11,Ru15} for relevant background, but omit a detailed discussion here.
As the characterization of preprojective and postinjective cluster variables in terms of the acyclic belt is what will be used in Section~\ref{sec:preproj}, we are content to take this as a definition in the skew-symmetrizable case.

\subsection{Acyclic cluster algebras from double Bruhat cells}
\label{sec:Coxeterdbc}

We can realize the cluster algebra associated to an acyclic skew-symmetrizable matrix $B$ geometrically in terms of the Kac-Moody group $G$ associated to its Cartan companion $A$.
More precisely, through the group we naturally find the algebra associated to the exchange matrix
\[
  \Bdp := \begin{bmatrix} B \\ Id_n \\ Id_n \end{bmatrix},
\]
which we say has \newword{doubled principal coefficients}.
When working with the cluster algebra $\cA_{\Bdp}$ we will denote the frozen variables $x_{i+n}$ and $x_{i+2n}$ for $1 \leq i \leq n$ by $\cvar_i$ and $\cvar_{\ol{\imath}}$, respectively.

We recall the double Bruhat decomposition
\[
  G = \bigsqcup_{u,v \in W} G^{u,v},\quad G^{u,v} : = B_+ u B_+ \cap B_- v B_-
\]
of $G$, where $B_{\pm}$ are the standard opposite Borel subgroups of $G$.
Each $G^{u,v}$ is a smooth affine variety of dimension $n + \ell(u) + \ell(v)$, where $\ell(u)$, $\ell(v)$ are the lengths of $u$ and $v$ \cite{FZ99,Wil13}.
The coordinate ring of $G^{u,v}$ is an upper cluster algebra and a finite subset of its clusters are in correspondence with shuffles of reduced words for $u$ and $v$ \cite{BFZ05,Wil13}.

\begin{remark}
  Each $G^{u,v}$ is the intersection of $G \subset \widehat{G}$ with the double Bruhat cell $\widehat{G}^{u,v}$ of $\widehat{G}$.
  The coordinate ring of the latter is also an upper cluster algebra whose initial exchange matrix has the same principal part, but with $\corank(A)$ additional frozen variables: $G^{u,v}$ is the locus where these are equal to one.
\end{remark}

\begin{theorem}
  \label{thm:coordring}
  Let $B$ be an acyclic skew-symmetrizable matrix, $G$ the Kac-Moody group of its Cartan companion $A$, and $c$ the associated Coxeter element.
  There is an isomorphism $\cA_{\Bdp} \cong \kk[G^{c,c^{-1}}]$ that identifies the initial cluster variables with restrictions of the following monomials in generalized minors:
  \begin{equation}
    \label{eq:initialcluster}
    x_i = \Delta_{\omega_i},
    \quad
    \cvar_i = \Delta^{\omega_i}_{ c\omega_i} \prod_{j < i}(\Delta^{\omega_j}_{c \omega_j})^{a_{ji}},
    \quad
    \cvar_{\ol{\imath}} = \Delta^{c \omega_i}_{\omega_i} \prod_{j < i}(\Delta^{c \omega_j}_{\omega_j})^{a_{ji}}.
  \end{equation}
\end{theorem}
\begin{proof}
  Specializing \cite[Theorem 2.10]{BFZ05} and \cite[Theorem 4.9]{Wil13} to the case at hand we get that $\kk[G^{c,c^{-1}}]$ is the upper cluster algebra with initial cluster variables $\Delta_{\omega_i}$, frozen variables $\Delta_{c\omega_i}^{\omega_i}$ and $\Delta_{\omega_i}^{c\omega_i}$, and initial exchange matrix obtained by extending $B$ with two copies of the matrix $Id_n-[B]_+$ (here we apply $[\,\cdot\,]_+$ entrywise).
  Note that \cite{BFZ05,Wil13} use transposed conventions on exchange matrices; here and everywere else in this paper we follow the convention of \cite{BFZ05}.

  This exchange matrix is acyclic by hypothesis and has full rank.
  Thus, following \cite[Corollary 1.19]{BFZ05}, $\kk[G^{c,c^{-1}}]$ is a cluster algebra.

  The cluster structure we obtain in this way is not quite the one in the statement of the theorem; the two are nonetheless strictly related.
  Indeed, since the variables $\cvar_i$ are obtained from the minors $\Delta^{\omega_j}_{ c\omega_j}$ by an invertible monomial transformation (and the variables $\cvar_{\ol{\imath}}$ are obtained similarly from the minors $\Delta^{c \omega_j}_{\omega_j}$), we can use separation of additions to rescale all the cluster variables in $\kk[G^{c,c^{-1}}]$ and establish the desired isomorphism (cf. \cite[Proposition 4.5]{YZ08}).
\end{proof}

  One can also consider the $2n$-dimensional subvariety $L^{c,c^{-1}} \subset G^{c,c^{-1}}$ where $z_{\ol{1}},\dotsc,z_{\ol{n}}$ are equal to one; this is the reduced double Bruhat cell studied in \cite{YZ08}.
  In particular, $\kk[L^{c,c^{-1}}]$ is the cluster algebra with principal coefficients associated to $B$, and our results immediately translate to statements in the principal coefficients case by specializing $\cvar_{\ol{1}},\dotsc,\cvar_{\ol{n}}$ to one.

\begin{remark}
There is another Lie-theoretic realization of $\cA_{\widetilde{B}}$ (for a certain extension of $B$) as the coordinate ring of $G^{e,c^2}$ \cite{GLS11}.
We do not know whether our main results describing cluster variables as minors also hold in some form for this realization.

\end{remark}

In computing various relations in $\kk[G^{u,v}]$ we will make extensive use of the fact that, if $u=s_{i_1}\cdots s_{i_{\ell(u)}}$ and $v=s_{j_1}\cdots s_{j_{\ell(v)}}$, then a generic $g \in G^{u,v}$ can be factored as
\begin{equation}
  \label{eq:uvfactorization}
  g = x_{\ol{\imath_1}}(t_{\ol{1}})\cdots x_{\ol{\imath_{\ell(u)}}}(t_{\ol{\ell(u)}})hx_{j_1}(t_{1}) \cdots x_{j_{\ell(v)}}(t_{\ell(v)})
\end{equation}
for some $h \in H$ and $t_j, t_{\ol{\imath}} \in \kk^\times$ \cite{FZ99,Wil13}.
For example, with respect to such a factorization it is clear that we have $\Delta_\lambda(g) = h^\lambda$ for any dominant weight $\lambda \in \widehat{P}^+$.
In particular, the generalized minors $\Delta_{\omega_i}$ that will be playing the role of initial (mutable) cluster variables evaluate to $h^{\omega_i}$.
Some similar computations are collected in Lemma~\ref{lemma:coefficients_values}.

\section{Preprojective and postinjective cluster variables as minors}
\label{sec:preproj}
In this section we prove that when $B$ is any acyclic skew-symmetrizable matrix, the identification $\cA_{\Bdp} \cong \kk[G^{c,c^{-1}}]$ realizes preprojective and postinjective cluster variables as restrictions of highest- and lowest-weight minors, respectively.
We first compute in Section~\ref{sec:acyclic_belt} the $\bfc$-vectors and $\bfg$-vectors appearing along the path \eqref{eq:path} as well as an explicit list of exchange relations that these cluster variables satisfy.
The identification with minors then follows in Section~\ref{sec:hlwminors} by comparing these relations with restrictions to $G^{c,c^{-1}}$ of the generalized determinantal identities of Proposition~\ref{prop:fundid}.

Recall that since $B$ is acyclic there exists a permutation $\sigma\in S_n$ such that $b_{\sigma_i,\sigma_j}\geq 0$ whenever $i<j$.
To simplify our notation in this section, we will assume that $\sigma$ is the identity permutation (no generality is lost, since we can always relabel the rows and columns of $B$ to achieve this).
Moreover, we specifically assume that the Cartan companion $A$ of $B$ is not of finite type, the main claims being already established in \cite{YZ08} for the case of a semisimple algebraic group.
Finally, without loss of generality, we will assume that $A$ is not block decomposable.

\subsection{Cluster variables in the acyclic belt}
\label{sec:acyclic_belt}

In view of the assumptions made at the beginning of this section, the Coxeter element associated to $B$ is $c=s_1\cdots s_n$ and the path (\ref{eq:path}) in the acyclic belt becomes
\begin{equation}
\label{eq:simplepath}
  \cdots
  \dashname{n}
  t_{-n}
  \dashname{1}
  \cdots
  \dashname{n-1}
  t_{-1}
  \dashname{n}
  t_0
  \dashname{1}
  t_1
  \dashname{2}
  \cdots
  \dashname{n}
  t_n
  \dashname{1}
  \cdots
\end{equation}
Moreover, any prefix of either of the infinite words $c^\infty =  s_1 \cdots s_n s_1 \cdots$ or $c^{-\infty} = s_n \cdots s_1 s_n \cdots$ is reduced (cf. \cite{Spe09}).
We will denote the prefix of $c^\infty$ of length $\Zidx$ by $c^\infty_{\le\Zidx}$ (so for example, $c^\infty_{\leq n}$ is just $c$ itself), similarly for $c^{-\infty}$.

For $\Zidx\in\ZZ$ we denote by $\mod{\Zidx}$ the unique integer in $[1,n]$ congruent to $\Zidx$ modulo $n$.
The edge immediately to the left of $t_\Zidx$ is then labelled by $\mod{\Zidx}$.
Observe that the nodes of $\TT$ we are considering lie along a \emph{sink/source adapted sequence}, i.e. every mutation occurs at a sink or source of the corresponding (valued) quiver $Q$.
More precisely, the $\mod{\Zidx}$-th column of $B^{t_\Zidx}$ contains only non-negative entries and the $\mod{\Zidx+1}$-st column only non-positive entries.
In particular, it follows from the definition of matrix mutation that the off-diagonal entries of the exchange matrix $B^{t_\Zidx}$ have the same absolute values as the corresponding entries of $A$.

Recall that, by definition, $\bfc$-vectors encode the coefficient part of the exchange relations in a cluster algebra with principal coefficients, and by extension also for doubled principal coefficients.
Thus to compute the exchange relations along the path \eqref{eq:simplepath} in $\cA_{\Bdp} \cong \kk[G^{c,c^{-1}}]$, we begin by describing $\bfc$-vectors.
Recall from Remark~\ref{rem:cgrootsweights} that we regard the $\bfc$-vectors $\cv_{i;t_\Zidx}$ as elements of the root lattice $\cQ$, the initial $\bfc$-vectors $\cv_{i;t_0}$ being identified with the simple roots $\alpha_i$.

\begin{lemma}
  \label{lemma:c-vectors}
  Let $\cv_{i;t_\Zidx}$ be the $i$-th $\bfc$-vector at vertex $t_\Zidx$.
  Then
  \begin{equation}
    \label{eq:c-vectors}
    \cv_{i;t_\Zidx}
    =
    \begin{cases}
      c^\infty_{\le\Zidx}\, \alpha_i & \text{if $\Zidx>0$;}\\
      \alpha_i & \text{if $-n < \Zidx \leq 0$ and  $i\leq n+\Zidx$;} \\
      -\alpha_i & \text{if $-n\leq \Zidx < 0$ and $i > n+\Zidx$;} \\
      -c^{-\infty}_{\le-\Zidx-n}\, \alpha_i \qquad & \text{if $\Zidx<-n$.}\\
    \end{cases}
  \end{equation}
  Moreover, whenever $\Zidx > 0$ or $\Zidx < -n$, the $\bfc$-vector $\cv_{\mod{\Zidx};t_\Zidx}$ is a negative root and the $\bfc$-vector $\cv_{\mod{\Zidx+1};t_\Zidx}$ is a positive root.
\end{lemma}

\begin{proof}
  We claim that for $\Zidx > 0$ the root $c^\infty_{\leq \Zidx}\, \alpha_{\mod{\Zidx+1}}$ is positive and the root $c^\infty_{\leq \Zidx}\, \alpha_{\mod{\Zidx}}$ is negative.
  Recall that for any reduced word $s_{i_1} \cdots s_{i_n}$, the root $s_{i_1} \cdots s_{i_{n-1}} \alpha_{i_n}$ is positive and the root $s_{i_1} \cdots s_{i_{n}} \alpha_{i_n}$ is negative \cite[VI \S 1.6, Corollary 2 of Proposition 17]{Bou02}.
  The positivity of $c^\infty_{\leq \Zidx}\, \alpha_{\mod{\Zidx+1}}$ then follows from the fact that  $c^\infty_{\leq \Zidx} s_{\mod{\Zidx+1}}$ is a reduced decomposition of $c^\infty_{\leq \Zidx+1}$, and the negativity of $c^\infty_{\leq \Zidx}\, \alpha_{\mod{\Zidx}}$ follows from the fact that $c^\infty_{\leq \Zidx-1}s_{\mod{\Zidx}}$ is a reduced decomposition of $c^\infty_{\leq \Zidx}$.
  Similarly, one checks that for $\Zidx < -n$ the root $-c^{-\infty}_{\le-\Zidx-n}\, \alpha_{\mod{\Zidx+1}}$ is positive and the root $-c^{-\infty}_{\le-\Zidx-n}\, \alpha_{\mod{\Zidx}}$ is negative.
  The second claim in the lemma will now follow once we prove the stated formulas for $\cv_{i;t_{\Zidx}}$.

  For the main claim there are three cases to consider.

  \noindent{\bf Case $-n \leq \Zidx \leq 0$:}
  We proceed by induction on $\Zidx$.
  For $\Zidx=0$ there is nothing to check.
  By construction, while the sign of $\cv_{\mod{\Zidx};t_\Zidx}$ is positive, the sign of the $\mod{\Zidx}$-th row of $B^{t_\Zidx}$ is negative.
  The $\bfc$-vectors at $t_{\Zidx-1}$ are therefore given by:
  \begin{equation*}
    \cv_{i;t_{\Zidx-1}}
    =
    \begin{cases}
      \cv_{i;t_\Zidx} & \text{if $i\neq \mod{\Zidx}$}\\
      -\cv_{\mod{\Zidx};t_\Zidx} & \text{if $i = \mod{\Zidx}$}
    \end{cases}
  \end{equation*}
  as desired.

  \noindent{\bf Case $\Zidx > 0$:}
  We again argue by induction on $\Zidx$ the base case still being $\Zidx=0$.
  Since the sign of $\cv_{\mod{\Zidx+1};t_\Zidx}$ and the sign of the $\mod{\Zidx+1}$-st row of $B^{t_\Zidx}$ are both positive, in view of \cite[Proposition 1.3]{NZ12}, the $\bfc$-vectors at $t_{\Zidx+1}$ are given by
  \begin{align*}
    \cv_{i;t_{\Zidx+1}}
    &=
    \cv_{i;t_\Zidx} - a_{\mod{\Zidx+1},i} \cv_{\mod{\Zidx+1};t_\Zidx}\\
    &=
    \cv_{i;t_\Zidx} - \frac{2(\alpha_{\mod{\Zidx+1}},\alpha_i)}{(\alpha_{\mod{\Zidx+1}},\alpha_{\mod{\Zidx+1}})} \cv_{\mod{\Zidx+1};t_\Zidx}\\
    &=
    \cv_{i;t_\Zidx} - \frac{2(\cv_{\mod{\Zidx+1};t_\Zidx},\cv_{i;t_\Zidx})}{(\cv_{\mod{\Zidx+1};t_\Zidx},\cv_{\mod{\Zidx+1};t_\Zidx})} \cv_{\mod{\Zidx+1};t_\Zidx}\\
    &=
    s_{\cv_{\mod{\Zidx+1};t_\Zidx}}\, \cv_{i;t_\Zidx},
  \end{align*}
  where the second to last equality holds because the pairing $(\cdot,\cdot)$ is invariant under the action of the Weyl group.
  Using the inductive hypotheses we can then conclude
  \[
    \cv_{i;t_{\Zidx+1}}
    =
    s_{\cv_{\mod{\Zidx+1};t_\Zidx}}\, \cv_{i;t_\Zidx}
    =
    c^\infty_{\le\Zidx} s_{\mod{\Zidx+1}} \left(c^\infty_{\le\Zidx}\right)^{-1} c^\infty_{\le\Zidx}\, \alpha_i
    =
    c^\infty_{\le\Zidx+1}\, \alpha_i.
  \]

  \noindent{\bf Case $\Zidx < -n$:}
  The argument is the same as the one used in the case $\Zidx>0$ with $\Zidx=-n$ as base for the induction.
  The only minor change is that this time we leverage the fact that both the sign of $\cv_{\mod{\Zidx};t_\Zidx}$ and the sign of the $\mod{\Zidx}$-th row of $B^{t_\Zidx}$ are negative.
\end{proof}

As a corollary of the preceding result,  we obtain the $\bfg$-vectors along the path \eqref{eq:simplepath} by applying Equation~\eqref{eq:g-vectors_from_c-vectors}.
Recall that we regard $\bfg$-vectors as elements of the weight lattice $P$, the initial $\bfg$-vectors $\gv_{i;t_0}$ being identified with the fundamental weights $\omega_i$.

\begin{lemma}
  \label{lemma:g-vectors}
  Let $\gv_{i;t_\Zidx}$ be the $i$-th $\bfg$-vector at $t_\Zidx$.
  Then
  \begin{equation}
    \label{eq:abeltgvectors}
    \gv_{i;t_\Zidx}
    =
    \begin{cases}
      c^\infty_{\le\Zidx}\, \omega_i & \text{if $0<\Zidx$;}\\
      \omega_i & \text{if $-n < \Zidx \leq 0$ and  $i\leq n+\Zidx$;} \\
      -\omega_i & \text{if $-n\leq \Zidx < 0$ and $i>n+\Zidx$;} \\
      -c^{-\infty}_{\le-\Zidx-n}\, \omega_i \qquad & \text{if $\Zidx <-n$.}\\
    \end{cases}
  \end{equation}
  In particular, the $\bfg$-vectors of the preprojective cluster variables are
  \[
    \Pi_{\textit{proj}}:=\big\{c^k\omega_i: i \in [1,n], k>0\big\},
  \]
  and the $\bfg$-vectors of the postinjective cluster variables are
  \[
  \Pi_{\textit{inj}}:=\big\{-c^{-k}\omega_i: i \in [1,n], k\geq0 \big\}.
  \]
\end{lemma}
\begin{proof}
  The result follows from Lemma~\ref{lemma:c-vectors} by leveraging the duality between the coroot and weight lattices for $A$.
  Indeed, Lemma~\ref{lemma:c-vectors} applied to $-B^T$ can be interpreted as saying that the $\bfc$-vectors for the dual cluster algebra are coroots of $A$; more precisely, at the node $t_\Zidx$, they are the coroots $\cv_{i;t_\Zidx}^\vee$.
  By Equation~\eqref{eq:g-vectors_from_c-vectors} the $\bfg$-vectors at $t_\Zidx$ are the basis of weights $\gv_{i;t_\Zidx}$ dual to $\cv_{i;t_\Zidx}^\vee$, hence the $\bfg$-vectors in Equation~\eqref{eq:abeltgvectors} are dual to the $\bfc$-vectors (for $-B^T$) in Equation~\eqref{eq:c-vectors}.

  The final claim now follows immediately from the fact that $s_i \omega_j = \omega_j$ unless $i =j$: for example, if $\Zidx > 0$ and $k$ is the smallest integer such that $kn \geq \Zidx$, we have $c^\infty_{\leq \Zidx} \omega_i =  c^k\omega_i$ if $\mod{\Zidx} \geq i$ and $c^\infty_{\leq \Zidx} \omega_i = c^{k-1}\omega_i$ if $\mod{\Zidx} < i$.
\end{proof}

With the preceding results in hand we are now ready to give recursive formulas for preprojective and postinjective cluster variables.
Recall that we denote by $x_\omega$ the cluster variable whose $\bfg$-vector is $\omega$.
We will use the shorthands $\beta_i^+:=s_1\cdots s_{i-1}\alpha_i$ and $\beta_i^-:=s_n\cdots s_{i+1}\alpha_i$, and denote by $[\beta:\alpha_i]$ the $i$-th coefficient in the expansion of a root $\beta$ in the basis of simple roots.
That is, we have
\[
  \beta = \sum_{i=1}^n [\beta:\alpha_i]\alpha_i
\]

\begin{theorem}
  \label{thm:vars_and_rels_in_bipartite_belt}
  The following exchange relations are satisfied in $\cA_{\Bdp}$:
  \begin{align}
    x_{c^k\omega_i}x_{c^{k+1}\omega_i}
    &=
    \prod_{j<i} x_{c^{k+1}\omega_j}^{-a_{ji}}
    \prod_{j>i} x_{c^k\omega_j}^{-a_{ji}}
    +
    \prod_{j=1}^n (z_j z_{\ol{\jmath}})^{[c^k\beta_i^+:\alpha_j]}
    \label{eq:preprojective}
    \\
    x_{-c^{-k}\omega_i}x_{-c^{-k-1}\omega_i}
    &=
    \prod_{j<i} x_{-c^{-k}\omega_j}^{-a_{ji}}
    \prod_{j>i} x_{-c^{-k-1}\omega_j}^{-a_{ji}}
    +
    \prod_{j=1}^n (z_j z_{\ol{\jmath}})^{[c^{-k}\beta_i^-:\alpha_j]}
    \label{eq:postinjective}
    \\
    x_{\omega_i}x_{-\omega_i}
    &=
    z_iz_{\ol{\imath}}
    \prod_{j<i} x_{\omega_j}^{-a_{ji}}
    \prod_{j>i} x_{-\omega_j}^{-a_{ji}}
    +
    1\label{eq:plusminus}
  \end{align}
  These relations determine all preprojective and postinjective cluster variables in terms of the initial ones.
\end{theorem}
\begin{proof}
  The relations are a direct consequence of Lemmas~\ref{lemma:c-vectors} and~\ref{lemma:g-vectors} once we observe that mutations happen along a sink/source adapted sequence, so that the $\mod{\Zidx}$-th column of $B^{t_\Zidx}$ contains only non-negative entries and the $\mod{\Zidx+1}$-th column only non-positive entries.
  The claim about the signs of the $\bfc$-vectors $\cv_{\mod{\Zidx};t_\Zidx}$ and $\cv_{\mod{\Zidx+1};t_\Zidx}$ in Lemma~\ref{lemma:c-vectors} pinpoints which monomial on the right hand side of each of these exchange relations contains coefficients.
\end{proof}
\begin{remark}
  As noted before, to obtain the exchange relations in the cluster algebra $\cA_{\Bpr}$ with principal coefficients it suffices to set the coefficients $\cvar_{\ol{1}},\dotsc,\cvar_{\ol{n}}$ to $1$ in the relations of Theorem~\ref{thm:vars_and_rels_in_bipartite_belt}.
\end{remark}

\subsection{Highest- and lowest-weight minors}
\label{sec:hlwminors}

Before proving the main result of the section, Theorem~\ref{thm:mainirregular}, we collect some lemmas concerning restrictions of minors to $G^{c,c^{-1}}$.
All of the arguments involved are straightforward generalizations of those in \cite{YZ08} to the infinite-type case.

The first lemma computes the evaluations of various minors of interest in terms of the factorization of a generic element of $G^{c,c^{-1}}$ into one-parameter subgroups.
Recall that the initial cluster variables in $\kk[G^{c,c^{-1}}]$ are restrictions of
\[
  x_i = \Delta_{\omega_i},
  \quad
  \cvar_i = \Delta^{\omega_i}_{ c\omega_i} \prod_{j < i}(\Delta^{\omega_j}_{c \omega_j})^{a_{ji}},
  \quad
  \cvar_{\ol{\imath}} = \Delta^{c \omega_i}_{\omega_i} \prod_{j < i}(\Delta^{c \omega_j}_{\omega_j})^{a_{ji}},
\]
for $1 \leq j \leq n$.
Recall also that we write $\beta_i^+:=s_1\cdots s_{i-1}\alpha_i$ and $\beta_i^-:=s_n\cdots s_{i+1}\alpha_i$, and denote by $[\beta:\alpha_i]$ the $i$-th coefficient in the expansion of a root $\beta$ in the basis of simple roots.

\begin{lemma}
  \label{lemma:coefficients_values}
  If $g \in G^{c,c^{-1}}$ factors as
  \begin{equation}
    \label{eq:generic_factorization}
    g = x_{\ol{1}}(t_{\ol{1}}) \cdots x_{\ol{n}}(t_{\ol{n}}) h x_n(t_n) \cdots x_1(t_1)
  \end{equation}
  for $h \in H$ and $t_i \in \kk^\times$, then
  \begin{gather}
    \Delta^{c^k\omega_j}_{c^{k+1}\omega_j}(g) = h^{c^k\omega_j}\prod_{i=1}^n t_{i}^{[c^k\beta_j^+:\alpha_i]},
    \quad
    \Delta^{c^{k+1}\omega_j}_{c^k\omega_j}(g) = h^{c^k\omega_j}\prod_{i=1}^n t_{\ol{\imath}}^{[c^k\beta_j^+:\alpha_i]},
    \label{eq:positive_cminor_values}\\
    \Delta^{-c^{-k-1}\omega_j}_{-c^{-k}\omega_j}(g) = h^{-c^{-k-1}\omega_j}\prod_{i=1}^n t_{i}^{[c^{-k}\beta_j^-:\alpha_i]},
    \quad
    \Delta^{-c^{-k}\omega_j}_{-c^{-k-1}\omega_j}(g) = h^{-c^{-k-1}\omega_j}\prod_{i=1}^n t_{\ol{\imath}}^{[c^{-k}\beta_j^-:\alpha_i]},
    \label{eq:negative_cminor_values}\\
    \cvar_i(g) = t_i h^{\omega_i}\prod_{ j <i}h^{a_{ j  i}\omega_ j },
    \quad
    \cvar_{\ol{\imath}}(g) = t_{\ol{\imath}} h^{\omega_i}\prod_{ j <i}h^{a_{ j  i}\omega_ j }.
    \label{eq:frozen_values}
  \end{gather}
\end{lemma}

\begin{proof}
  The first four identities follow easily from considering root strings in $V(\omega_j)$ as in \cite[Proof of Equation 3.20]{YZ08}, and we omit an explicit discussion.
  For example, if $v_{\omega_j}$ is a highest-weight vector of $V(\omega_j)$, then $\Delta^{c\omega_j}_{\omega_j}(g)$ measures the component of $gv_{\omega_j}$ in the weight space $V(\omega_j)_{c \omega_j}$; sequentially considering how each one-parameter subgroup in the factorization can affect this component leads to the given formula.

  We derive the first identity of Equation~\eqref{eq:frozen_values}, the second being similar.
  Plugging Equation~\eqref{eq:positive_cminor_values} into the definition of $\cvar_i(g)$ we obtain
  \[
    \cvar_i(g)
    =
    \left(h^{\omega_i}\prod_{\ell =1}^n t_{\ell}^{[\beta_i^+:\alpha_\ell]}\right)
    \prod_{j<i}\left(h^{\omega_j}\prod_{\ell=1}^n t_{\ell}^{[\beta_j^+:\alpha_\ell]}\right)^{a_{ji}}.
  \]
  To establish the desired identity we need to verify that
  \[
    [\beta_i^+:\alpha_\ell]+\sum_{j<i}a_{ji}[\beta_j^+:\alpha_\ell]
    =
    \begin{cases}
      1 & \text{ if $i=\ell$}\\
      0 & \text{ otherwise.}
    \end{cases}
  \]
  This in turn follows from
  \begin{equation}
    \label{eq:alpha_to_beta}
    \beta_i^++\sum_{j<i}a_{ji}\beta_j^+ = s_1\cdots s_{i-1}\alpha_i + \sum_{j<i}s_1\cdots s_{j-1}(\alpha_i - s_j\alpha_i)
    =
    \alpha_i,
  \end{equation}
  with the last equality given by collapsing the telescoping sum.
\end{proof}

\begin{lemma}
  \label{lemma:coefficient_identity_powers}
  For $g \in G^{c,c^{-1}}$, $k \geq 0$, and $1 \leq i \leq n$ we have
  \begin{equation*}
    \Delta_{c^{k+1} \omega_i}^{c^k \omega_i}(g)
    =
    \prod_{j=1}^nz_j(g)^{[c^k\beta_i^+:\alpha_j]},\quad
    \Delta_{c^k \omega_i}^{c^{k+1} \omega_i}(g)
    =
    \prod_{j=1}^nz_{\ol{\jmath}}(g)^{[c^k\beta_i^+:\alpha_j]}
  \end{equation*}
\end{lemma}
\begin{proof}
  We derive the first identity, the second being obtained in a similar way.
  It suffices to consider the case where $g$ factors as in Equation~\eqref{eq:generic_factorization}, since such elements are generic in $G^{c,c^{-1}}$.
  In the parameters of the factorization, Lemma~\ref{lemma:coefficients_values} lets us rewrite the desired identity as
  \[
    h^{c^k\omega_i}\prod_{j=1}^n t_{j}^{[c^k\beta_i^+:\alpha_j]}
    =
    \prod_{j=1}^n\Big( t_j h^{\omega_j}\prod_{ \ell <j}h^{a_{\ell j}\omega_ \ell }\Big)^{[c^k\beta_i^+:\alpha_j]}.
  \]
  Establishing this is in turn equivalent to showing that in $P$ we have
  \begin{equation}
    \label{eq:coxeter_identity_omega}
    c^k\omega_i
    =
    \sum_{j=1}^n
    [c^k\beta_i^+:\alpha_j]
    \Big(
    \omega_j+
    \sum_{\ell<j} a_{\ell j}\omega_\ell
    \Big).
  \end{equation}

  On the other hand, we have
  \begin{equation}
    \label{eq:coxeter_identity_beta}
    c^k\beta_i^+
    =
    \sum_{j=1}^n [c^k\beta_i^+:\alpha_j] \alpha_j
    =
    \sum_{j=1}^n [c^k\beta_i^+:\alpha_j] \Big(\beta_j^+ + \sum_{\ell<j} a_{\ell j}\beta_\ell^+\Big),
  \end{equation}
  where the second equality follows from Equation~\eqref{eq:alpha_to_beta}.
  Observe that the endomorphism $1-c$ of $\widehat{P}$, in addition to taking $\omega_i$ to $\beta_i^+$, factors through the natural projection $\widehat{P} \onto P$.
  Since the $\beta_i^+$ are linearly independent and generate the root lattice $\cQ$ (they are obtained from the $\alpha_i$ by a unitriangular transformation), $1-c$ restricts to an isomorphism of $P$ onto $\cQ$; with some abuse we denote its inverse as $(1-c)^{-1}: \cQ \to P$.
  Since $1-c$ commutes with the action of $c^{k}$ so does $(1-c)^{-1}$, hence applying $(1-c)^{-1}$ to Equation~\eqref{eq:coxeter_identity_beta} yields Equation~\eqref{eq:coxeter_identity_omega}.
  Note that Equation~\eqref{eq:coxeter_identity_omega} does not necessarily hold if interpreted in $\widehat{P}$, since in general a splitting $P \into \widehat{P}$ is not $W$-equivariant.
\end{proof}

\begin{lemma}
  \label{lemma:coefficients_identities}
  For $g \in G^{c,c^{-1}}$ and $1 \leq j \leq n$ we have
  \begin{equation*}
    \cvar_j(g)
    =
    \Delta_{-\omega_j}^{-c^{-1}\omega_j}(g)
    \prod_{\ell>j}\Big(\Delta_{-\omega_\ell}^{-c^{-1}\omega_\ell}(g)\Big)^{a_{\ell j}},
    \quad
    \cvar_{\ol{\jmath}}(g)
    =
    \Delta_{-c^{-1}\omega_j}^{-\omega_j}(g)
    \prod_{\ell>j}\Big(\Delta_{-c^{-1}\omega_\ell}^{-\omega_\ell}(g)\Big)^{a_{\ell j}}.
  \end{equation*}
\end{lemma}
\begin{proof}
  Again, we derive the first identity, the second being obtained in a similar fashion.
  If $g$ is generic and factored as in Equation~\eqref{eq:generic_factorization}, then in terms of factorization parameters the desired identity is
  \begin{equation}
    t_j  h^{\omega_j} \prod_{\ell<j}h^{a_{\ell j}\omega_\ell}
    =
    t_j  h^{-c^{-1}\omega_j} \prod_{\ell>j}h^{-a_{\ell j}c^{-1}\omega_\ell}.
  \end{equation}
  In rewriting the left-hand side we have used Equation~\eqref{eq:frozen_values}, while rewriting the right-hand side uses a similar computation involving the identity
  \[
    \beta_j^-+\sum_{\ell>j}a_{\ell j}\beta_\ell^-
    =
    \alpha_j.
  \]

  This equality of monomials in factorization parameters is then equivalent to the equality
  \begin{equation}
    \label{eq:c_inverse_equality}
    \omega_j + \sum_{\ell<j} a_{\ell j}\omega_\ell
    =
    -c^{-1}\left(\omega_j + \sum_{\ell>j} a_{\ell j}\omega_\ell\right)
  \end{equation}
  of weights in $P$.

  Multiplying both sides by $s_{j+1}\cdots s_n$ transforms Equation~\eqref{eq:c_inverse_equality} into
  \[
    s_{j+1}\cdots s_n\left(\omega_j + \sum_{\ell<j} a_{\ell j}\omega_\ell\right)
    =
    -s_js_{j-1}\cdots s_1\left(\omega_j + \sum_{\ell>j} a_{\ell j}\omega_\ell\right).
  \]
  Simplifying using the fact that $s_i\omega_j=\omega_j$ for $i\neq j$, then moving both summations to the right-hand side turns this into
  \begin{equation}
    s_j\omega_j
    =
    -\omega_j-\sum_{\ell\neq j}a_{\ell j}\omega_\ell
    =
    \omega_j-\alpha_j,
  \end{equation}
  which is the definition of $s_j\omega_j \in P$.
  Note that for this equality to hold when interpreted in $\widehat{P}$, additional terms in the kernel of $\widehat{P} \onto P$ must in general be added to the right-hand side.
\end{proof}

In proving the following theorem we make use of the involutive anti-automorphism $g \mapsto g^\iota$ of $G$ defined by
\begin{equation}
  h^\iota := h^{-1}\,\, (h\in H),
  \quad
  \quad
  x_i(t)^\iota := x_i(t),
  \quad
  \quad
  x_{\ol{\imath}}(t)^\iota := x_{\ol{\imath}}(t).
\end{equation}

Since $\iota$ preserves $B_+$ and $B_-$ but takes any lift of $w \in W$ to a lift of $w^{-1}$, it restricts to a biregular isomorphism between $G^{u,v}$ and $G^{u^{-1},v^{-1}}$.
Since a generic element of $G$ factors as $n_- h n_+$, with $n_- \in N_-$, $h \in H$, $n_+ \in N_+$, we have $\Delta_{\omega_j}(g) = \Delta_{-\omega_j}(g^\iota)$ for all $g \in G$.

\begin{theorem}
  \label{thm:mainirregular}
  The isomorphism $\cA_{\Bdp} \cong \kk[G^{c,c^{-1}}]$ of Theorem~\ref{thm:coordring} identifies the preprojective cluster variable $x_{c^k\omega_i}$ with the restriction of the highest-weight minor $\Delta_{c^k\omega_i}$ and identifies the postinjective cluster variable $x_{-c^k \omega_i}$ with the restriction of the lowest-weight minor $\Delta_{-c^k \omega_i}$.
\end{theorem}

\begin{proof}
  In Theorem~\ref{thm:vars_and_rels_in_bipartite_belt} we explicitly computed exchange relations which allow one to compute these cluster variables by iterated mutation from the initial variables, so it suffices to check that the indicated restrictions of minors satisfy the same relations.

  If we apply Proposition~\ref{prop:fundid} with $u = v = c^k s_1 \cdots s_{i-1}$ and substitute using the equalities from Lemma~\ref{lemma:coefficient_identity_powers}, we find that
  \begin{equation}
    \Delta_{c^k\omega_i}\Delta_{c^{k+1}\omega_i}
    =
    \prod_{j<i}\Delta_{c^{k+1}\omega_j}^{-a_{ji}}
    \prod_{j>i}\Delta_{c^k\omega_j}^{-a_{ji}}
    +
    \prod_{j=1}^n(z_j z_{\ol{\jmath}})^{[c^k\beta_i^+:\alpha_j]}
  \end{equation}
  is satisfied on $G^{c,c^{-1}}$.
  Comparing with Equation~\eqref{eq:preprojective} establishes that $x_{c^k\omega_i} = \Delta_{c^k \omega_i}$.

  On the other hand, applying the above argument with $c^{-1}$ in place of $c$ shows that
  \begin{equation*} \Delta_{c^{-k}\omega_i} \Delta_{c^{-k-1}\omega_i} =
    \prod_{j<i}\Delta_{c^{-k}\omega_j}^{-a_{ji}}
    \prod_{j>i}\Delta_{c^{-k-1}\omega_j}^{-a_{ji}}
    +
    \prod_{j=1}^n\left(
      \Delta^{c^{-1} \omega_j}_{\omega_j}
      \Delta^{\omega_j}_{ c^{-1}\omega_j}
      \prod_{\ell>j}\big(
        \Delta^{c^{-1} \omega_\ell}_{\omega_\ell}
        \Delta^{\omega_\ell}_{c^{-1} \omega_\ell}
      \big)^{a_{\ell j}}
    \right)^{[c^{-k}\beta_i^-:\alpha_j]}
  \end{equation*}
  is satisfied on $G^{c^{-1},c}$.
  Given that $g \mapsto g^\iota$ exchanges $G^{c,c^{-1}}$ with $G^{c^{-1},c}$ and that $\Delta_\lambda^\mu(g) = \Delta_{-\mu}^{-\lambda}(g^\iota)$, we obtain another identity on $G^{c,c^{-1}}$.
  Using the substitution of Lemma~\ref{lemma:coefficients_identities}, this can be rewritten as
  \begin{equation}
  \Delta_{-c^{-k}\omega_i}\Delta_{-c^{-k-1}\omega_i} =
    \prod_{j<i}\Delta_{-c^{-k}\omega_j}^{-a_{ji}}
    \prod_{j>i}\Delta_{-c^{-k-1}\omega_j}^{-a_{ji}}
    +
    \prod_{j=1}^n (z_j z_{\ol{\jmath}})^{[c^{-k}\beta_i^-:\alpha_j]}.
  \end{equation}
  In particular, the lowest-weight minors $\Delta_{-c^{-k}\omega_i}$ satisfy Equation~\eqref{eq:postinjective}.

  To establish that the indicated minors satisfy Equation~\eqref{eq:plusminus} and thereby conclude the proof, we must show that
  \begin{equation*}
    \Delta_{\omega_i}\Delta_{-\omega_i}
    =
    \cvar_i \cvar_{\ol{\imath}}
    \prod_{j<i}\Delta_{\omega_j}^{-a_{ji}}
    \prod_{j>i}\Delta_{-\omega_j}^{-a_{ji}}
    +
    1
  \end{equation*}
  is satisfied on $G^{c,c^{-1}}$.
  Evaluating both sides on a generic element factored as in Equation~\eqref{eq:generic_factorization}, we apply Lemma~\ref{lemma:coefficients_values} to obtain
  \begin{align*}
    h^{\omega_i}\Delta_{-\omega_i}(g)
    &=
    t_i
    t_{\ol{\imath}}
    h^{2\omega_i}
    \prod_{j<i}h^{a_{j i}\omega_j}
    \prod_{j>i}\Delta_{-\omega_j}^{-a_{ji}}(g)
    +
    1
    =
    t_i
    t_{\ol{\imath}}
    h^{\alpha_i}
    \prod_{j>i}\Big(h^{\omega_j}\Delta_{-\omega_j}(g)\Big)^{-a_{ji}}
    +
    1.
  \end{align*}
  Applying $\iota$, this is equivalent to
  \begin{equation*}
    h^{\omega_i}\Delta_{\omega_i}(g^\iota) =
    t_i
    t_{\ol{\imath}}
    h^{\alpha_i}
    \prod_{j>i}\Big(h^{\omega_j}\Delta_{\omega_j}(g^\iota)\Big)^{-a_{ji}}
    +
    1.
  \end{equation*}

  To compute $\Delta_{\omega_i}(g^\iota)$ we need to straighten up the expression of $g^\iota$ using the fact that
  \begin{equation*}
    x_i(t_i) h x_{\ol{\imath}}(t_{\ol{\imath}})
    =
    x_{\ol{\imath}}(t'_{\ol{\imath}})
    h
    (1+t_it_{\ol{\imath}}h^{-\alpha_i})^{\alpha_i^\vee}
    x_{i}(t'_i),
  \end{equation*}
  for some $t'_i$, $t'_{\ol{\imath}} \in \kk^\times$ \cite[Proposition 7.2]{BZ01} (this imposes an irrelevant genericity condition on $t'_i$ and $t'_{\ol{\imath}}$).
  Using this and the fact that $x_i(t_i)$ and $x_{\ol{\jmath}}(t_{\ol{\jmath}})$ commute for $i \neq j$, a straightforward computation yields
  \[
    g^\iota
    =
    x_{\ol{n}}(t'_{\ol{n}})
    \cdots
    x_{\ol{1}}(t'_{\ol{1}})
    h^{-1}
    v_1^{\alpha^\vee_1}
    \dots
    v_n^{\alpha^\vee_n}
    x_1(t'_1)
    \dots
    x_n(t'_n).
  \]
  Here the $v_i$ are defined by the recurrence relation
  \[
    v_i
    =
    t_i t_{\ol{\imath}} h^{\alpha_i}
    \prod_{j > i} v_j^{-a_{j,i}}+1,
  \]
  and are non-zero elements of $\kk$ for generic values of the parameters $t_i$ and $t_{\ol{\imath}}$.
  From this one observes that $h^{\omega_i}\Delta_{\omega_i}(g^\iota) = v_i$, and the desired identity follows.
\end{proof}

\section{Regular cluster variables as minors}

In this section we prove that the identification $\cA_{\Bdp} \cong \kk[G^{c,c^{-1}}]$ realizes all regular cluster variables as restrictions of level zero minors in the case that the Cartan companion of $B$ is of type $A_{n-1}^{\!(1)}$.
The proof exploits the fact that, when $G$ is the central extension $\widetilde{LSL}_n$ of the loop group of $SL_n$, the level zero minors can be expressed combinatorially as weighted sums over collections of paths in a directed network on a cylinder.
We show that the collections of paths relevant to a given minor are in bijection with submodules of a corresponding regular quiver representation, and from this that the network computation directly recovers the Caldero-Chapoton formula.
A similar computation shows that the generic basis element associated to the base of any homogeneous tube in type $A_{n-1}^{\!(1)}$ is also a level zero minor.
We expect that all regular cluster variables are minors of level zero representations for any affine type, and we verify this for the finite list of types where it can be checked using only the characters of those representations.

\subsection{Combinatorics of regular representations}
\label{sec:rigidregular}
Let $Q$ be an acyclic orientation of an $n$-cycle, and label the vertices of $Q$ by $\ZZ_n$ so that neighboring vertices have cyclically adjacent labels.
We have the following explicit classification of the indecomposable regular representations $\Qrep$ which are rigid (i.e. $\Ext^1(\Qrep,\Qrep)=0$).
For $k,\ell\in\ZZ_n$ we write $[k,\ell]:=\{k,k+1,\ldots,\ell\}$ and let $\Qrep_{[k,\ell]}$ be the representation with
\[
  \Qrep_i \cong \begin{cases} \kk & i \in [k,\ell] \\ 0 & \text{otherwise,} \end{cases}
\]
and all arrows acting by invertible maps except possibly the arrow between $k-1$ and $k$ and the arrow between $\ell$ and $\ell + 1$.

\begin{proposition}
  If $[k,\ell] \subsetneq \ZZ_n$ is a proper subset containing the same number of sinks and sources of $Q$, then $\Qrep_{[k,\ell]}$ is an indecomposable rigid regular representation of $Q$.
  Every such representation is isomorphic to one of this form.
\end{proposition}
\begin{proof}\mbox{}
  We first recall the structure of the full subcategory $\cR\subset\rep_\kk(Q)$ of regular representations \cite[Chapter 12, Section 2]{ASS06}.
  Say an arrow $a$ of $Q$ is \newword{clockwise} if $s(a) = t(a) - 1$ and \newword{counterclockwise} if $s(a) = t(a) + 1$.
  Let $p$ and $q$ be the numbers of clockwise and counterclockwise arrows, respectively.
  Then we have a decomposition $$\cR \cong \cR_{nil}^p \oplus \cR_{nil}^q \oplus \cR_{inv},$$ where $\cR_{nil}^p$ (resp. $\cR_{nil}^q$) is equivalent to the category of nilpotent representations of an oriented $p$-cycle (resp. $q$-cycle), and $\cR_{inv}$ is equivalent to the category of invertible representations of the Jordan quiver (this is the quiver with a single vertex and a single arrow from that vertex to itself, and an invertible representation is one for which this arrow acts by an invertible linear map).

  Any invertible representation of the Jordan quiver has nontrivial self-extensions: if $X$ is an invertible matrix defining such a representation $\Qrep$, then the block upper triangular matrix
  \[
    \begin{bmatrix} X & Id \\ 0 & X \end{bmatrix}
  \]
  defines a self-extension of $\Qrep$.
  In particular, $\cR_{inv}$ does not contain any nontrivial rigid representations.

  Now consider the category $\cC^p$ of nilpotent representations of an oriented $p$-cycle.
  Write $\tilde\cC$ for the category of finite-dimensional representations of the equi-oriented $A_{\infty}$ quiver (i.e. the quiver with vertex set $\ZZ$, and having one arrow from $i$ to $i+1$ for all $i \in \ZZ$).
  For any choice of base point, there is an essentially surjective exact functor $\pi:\tilde\cC\to\cC^p$ given by ``wrapping around'' the $p$-cycle (nilpotency guarantees the essential surjectivity).
  Now consider an indecomposable object $\Qrep$ of $\cC^p$ which is sincere (i.e. $\Qrep_i \neq 0$ for every vertex $i$).
  Given any lift $\tilde \Qrep\in\tilde\cC$ of $\Qrep$, let $\tilde \Qrep'$ be the lift of $\Qrep$ obtained by shifting $\tilde \Qrep$ in the direction of the arrows by $p$ vertices.
  Since $\Qrep$ was sincere, $\tilde\cC$ contains a nontrivial extension of $\tilde \Qrep$ by $\tilde \Qrep'$.
  Since $\pi$ is exact, the image of this extension is a nontrivial self-extension of $\Qrep$ in $\cC^p$.
  As all counterclockwise arrows act by invertible linear maps on any $Q$-representation in $\cR_{nil}^p$, the equivalence $\cC^p \cong \cR_{nil}^p$ takes sincere representations to sincere representations.
  In particular, we conclude that no sincere representation in $\cR_{nil}^p$ is rigid, and likewise for $\cR_{nil}^q$ by the same argument.

  From the discussion so far we see that any indecomposable rigid regular representation of $Q$ is non-sincere and in either $\cR_{nil}^p$ or $\cR_{nil}^q$.
  Any proper subquiver $Q'$ of $Q$ defines an exact inclusion $\rep_\kk(Q') \into \rep_\kk(Q)$.
  Any non-sincere representation of $Q$ is the image of a representation of a proper subquiver, which will necessarily be of type $A_r$ for some $r < n$.
  Thus any indecomposable rigid regular representation of $Q$ is the image of an indecomposable representation of such a subquiver (these will all be rigid) -- in particular, it is isomorphic to $\Qrep_{[k,\ell]}$ for a proper subset $[k,\ell]\subsetneq\ZZ_n$.

  It remains to derive the condition on sinks and sources in $[k,\ell]$.
  We identify the Grothendieck group $\KQ$ with the root lattice $\cQ$ of type $A_{n-1}^{\!(1)}$, and write $\delta\in\cQ$ for the primitive positive imaginary root.
  Recall that $\Qrep \in \rep_\kk(Q)$ is regular if and only if the Euler-Ringel pairing $\langle\delta,\Qrep\rangle$ vanishes \cite[Section 7]{CB92}.
  Writing $I \subset \ZZ_n$ (resp. $O \subset \ZZ_n$) for the subset of sinks (resp. sources) in $Q$, we have
  \[
    \langle\delta,\Qrep\rangle
    =
    \sum_{j\in Q_0}\dim \Qrep_j-\sum_{a\in Q_1}\dim \Qrep_{t(a)}
    =
    \sum_{o\in O}\dim \Qrep_o-\sum_{i\in I}\dim \Qrep_i.
  \]
  Thus $\Qrep$ is regular if and only if $\sum\limits_{o\in O}\dim \Qrep_o=\sum\limits_{i\in I}\dim \Qrep_i$.
  In particular, this holds for $\Qrep = \Qrep_{[k,\ell]}$ exactly when the number of sinks in $[k,\ell]$ is equal to the number of sources in $[k,\ell]$.
\end{proof}

With this classification in hand, we now describe the cluster characters of the rigid regular representations of $Q$ combinatorially in terms of target-closed subsets.
Given $[k,\ell]\subsetneq\ZZ_n$, we say a subset $E$ of $[k,\ell]$ is \newword{target-closed} if whenever $a$ is an arrow of $Q$ with $s(a) \in E$ and $t(a) \in [k,\ell]$, then in fact $t(a) \in E$.
We let $\cE_{[k,\ell]}$ denote the collection of target-closed subsets of $[k,\ell]$, and otherwise follow the notation of Section~\ref{sec:quiverbackground}.

\begin{proposition}
  \label{prop:regular coindices}
  Let $[k,\ell] \subsetneq \ZZ_n$ be a proper subset containing the same number of sinks and sources of $Q$, so $\Qrep_{[k,\ell]}$ is rigid and regular.
  If $i_1 < \cdots < i_m$ (resp. $o_1 < \cdots < o_m$) are the sinks (resp. sources) of $Q$ contained in $[k,\ell]$ (ordered by the restriction of the cyclic order on $\ZZ_n$ to a linear order on $[k,\ell]$), then
  \[
    \grep(\Qrep_{[k,\ell]})
    =
    \begin{cases}
      \omega_{k-1}-\omega_\ell+\sum_{j=1}^m (\omega_{o_j}-\omega_{i_j}) & \text{$i_1<o_1$ in $[k,\ell]$}\\
      \omega_{\ell+1}-\omega_k+\sum_{j=1}^m (\omega_{o_j}-\omega_{i_j}) & \text{$o_1 < i_1$ in $[k,\ell]$,}
    \end{cases}
  \]
  and the (framed) cluster character of $\Qrep_{[k,\ell]}$ is
  \begin{equation}
    \label{eq:regular cluster characters}
    x_{\Qrep_{[k,\ell]}}
    =
    x^{\grep(\Qrep_{[k,\ell]})} \sum_{E \in \cE_{[k,\ell]}} \prod_{j \in E} \hat{y}_j.
  \end{equation}
\end{proposition}
\begin{proof}
  If $i_1 < o_1$ in $[k,\ell]$, then it is straightforward to see that the minimal injective resolution of $\Qrep_{[k,\ell]}$ is of the form
  \[
    0
    \to
    \Qrep_{[k,\ell]}
    \to
    I_{i_1} \oplus \cdots \oplus I_{i_m}\oplus I_\ell
    \to
    I_{k-1} \oplus I_{o_1} \oplus \cdots \oplus I_{o_m},
  \]
  from which the given formula for $\grep(\Qrep_{[k,\ell]})$ follows.
  The $o_1 < i_1$ case is similar.
  The description of submodules of $\Qrep_{[k,\ell]}$ in terms of target-closed subsets of $[k,\ell]$ is an instance of a more general description in terms of the coefficient quiver of a representation \cite{CI11,Rin98}; in our case the coefficient quiver of $\Qrep_{[k,\ell]}$ is just the full subquiver of $Q$ with vertex set $[k,\ell]$.
\end{proof}

In Section~\ref{sec:networks} it will be useful to rewrite regular $\bfg$-vectors using the following lemma, which follows from Proposition~\ref{prop:regular coindices} by telescoping the sum below.

\begin{lemma}
  \label{lem:regulargvectors}
  If $\Qrep_{[k,\ell]}$ is rigid and regular, we can write its $\bfg$-vector as
  \[
    \grep(\Qrep_{[k,\ell]})=\sum_{j\in S_{[k,\ell]}} \omega_{j+1} - \omega_j,
  \]
  where
  \[
    S_{[k,\ell]}
    =
    \begin{cases}
      [\ell,k-2] \cup [i_1, o_1 -1] \cup \cdots \cup [i_m, o_m - 1] & \text{$i_1<o_1$ in $[k,\ell]$}\\
      [k,o_1-1] \cup [i_1,o_2-1] \cup \cdots \cup [i_m,\ell] & \text{$o_1 < i_1$ in $[k,\ell]$.}
    \end{cases}
  \]
\end{lemma}

\subsection{Level zero minors via networks}
\label{sec:networks}
Recall that the generic element $g$ of any double Bruhat cell $G^{u,v}$ can be factorized as in Equation~\eqref{eq:uvfactorization}.
When $G = \widetilde{LSL}_n$ is the central extension of $LSL_n$, these factorizations (and more general ones associated to shuffles of reduced words) admit a combinatorial description in terms of directed networks on a cylinder \cite{GSV12,FM14}.
This generalizes a similar relation between $SL_n$ and directed networks on a disk \cite{FZ99}.
By a directed network on a surface we will mean an embedded directed graph whose edges are assigned weights in some fixed ring.
As in Section~\ref{sec:rigidregular} we fix a cyclic $\ZZ_n$-labeling of the vertices of the $A_{n-1}^{\!(1)}$ Dynkin diagram, chosen so that $0 \equiv n$ is the label of the affine root $\delta - \theta$.

\begin{figure}
	\centering
  \label{fig:network}
  \begin{tikzpicture}
    \node[below] at (0,-2) {$x_{\ol{\imath}}(t)$};
    \draw (0.75,0) ellipse (0.5 and 2);
    \draw (-0.75,2) arc (90:270:0.5 and 2);
    \draw[dashed] (-0.75,2) arc (90:-90:0.5 and 2);
    \draw[-] (0.75,2) -- (-0.75,2);
    \draw[-] (0.75,-2) -- (-0.75,-2);
    \draw[-] (1.25,0) -- (0.25,0);
    \draw[dashed] (-0.25,0) -- (0.21,0);
    \draw[singlearrow] (1.25,0) -- (-0.25,0);
    \node[right] at (1.25,0) {$\scriptstyle 1$};
    \draw[-] (1.2,-0.85) -- (0.3,-0.85);
    \draw[dashed] (-0.3,-0.85) -- (0.3,-0.85);
    \draw[singlearrow] (1.2,-0.85) -- (-0.3,-0.85);
    \node[right] at (1.2,-0.85) {$\scriptstyle n$};
    \draw[-] (1.015,-1.7) -- (0.49,-1.7);
    \draw[dashed] (-0.49,-1.7) -- (0.45,-1.7);
    \draw[singlearrow] (1.015,-1.7) -- (-0.49,-1.7);
    \node[right] at (1.015,-1.7) {$\scriptstyle n-1$};
    \draw[-] (-1.225,-0.65) -- (0.275,-0.65);
    \draw[doublearrow,draw] (0.275,-0.65) -- (-1.225,-0.65);
    \node[right] at (0.275,-0.65) {$\scriptstyle i$};
    \draw[-] (-1.245,0.2) -- (0.255,0.2);
    \draw[doublearrow,draw] (0.255,0.2) -- (-1.245,0.2);
    \node[right] at (0.255,0.2) {$\scriptstyle i-1$};
    \draw[singlearrow,draw] (-0.495,0.2) arc (174.2:199:0.5 and 2);
    \node at (-0.8,-0.225) {$\scriptstyle t$};
    \draw[-] (1.015,1.7) -- (0.49,1.7);
    \draw[dashed] (-0.49,1.7) -- (0.45,1.7);
    \draw[singlearrow] (1.015,1.7) -- (-0.49,1.7);
    \node[right] at (1.015,1.7) {$\scriptstyle 3$};
    \draw[-] (1.2,0.85) -- (0.3,0.85);
    \draw[dashed] (-0.3,0.85) -- (0.3,0.85);
    \draw[singlearrow] (1.2,0.85) -- (-0.3,0.85);
    \node[right] at (1.2,0.85) {$\scriptstyle 2$};
  \end{tikzpicture}
  \quad\quad
  \begin{tikzpicture}
    \node[below,color=white] at (0,-2) {$()$};
    \node[below] at (0,-2) {$h$};
    \draw (0.75,0) ellipse (0.5 and 2);
    \draw (-0.75,2) arc (90:270:0.5 and 2);
    \draw[dashed] (-0.75,2) arc (90:-90:0.5 and 2);
    \draw[-] (0.75,2) -- (-0.75,2);
    \draw[-] (0.75,-2) -- (-0.75,-2);
    \draw[-] (1.25,0) -- (0.25,0);
    \draw[dashed] (-0.25,0) -- (0.21,0);
    \draw[singlearrow] (1.25,0) -- (-0.25,0);
    \node[right] at (1.25,0) {$\scriptstyle 1$};
    \draw[-] (1.2,-0.85) -- (0.3,-0.85);
    \draw[dashed] (-0.3,-0.85) -- (0.3,-0.85);
    \draw[singlearrow] (1.2,-0.85) -- (-0.3,-0.85);
    \node[right] at (1.2,-0.85) {$\scriptstyle n$};
    \draw[-] (1.015,-1.7) -- (0.49,-1.7);
    \draw[dashed] (-0.49,-1.7) -- (0.45,-1.7);
    \draw[singlearrow] (1.015,-1.7) -- (-0.49,-1.7);
    \node[right] at (1.015,-1.7) {$\scriptstyle n-1$};
    \draw[-] (-1.225,-0.65) -- (0.275,-0.65);
    \draw[doublearrow] (0.275,-0.65) -- (-1.225,-0.65);
    \node[right] at (0.275,-0.65) {$\scriptstyle i$};
    \draw[thick,color=white] (-0.255,-0.28) arc (-8:-18.5:0.5 and 2);
    \node at (-0.5,-0.435) {$\scriptstyle h_{i+1}h_i^{-1}$};
    \draw[-] (-1.245,0.2) -- (0.255,0.2);
    \draw[doublearrow] (0.255,0.2) -- (-1.245,0.2);
    \node[right] at (0.255,0.2) {$\scriptstyle i-1$};
    \draw[thick,color=white] (-0.255,0.22) arc (6.25:18:0.5 and 2);
    \node at (-0.5,0.405) {$\scriptstyle h_ih_{i-1}^{-1}$};
    \draw[-] (1.015,1.7) -- (0.49,1.7);
    \draw[dashed] (-0.49,1.7) -- (0.45,1.7);
    \draw[singlearrow] (1.015,1.7) -- (-0.49,1.7);
    \node[right] at (1.015,1.7) {$\scriptstyle 3$};
    \draw[-] (1.2,0.85) -- (0.3,0.85);
    \draw[dashed] (-0.3,0.85) -- (0.3,0.85);
    \draw[singlearrow] (1.2,0.85) -- (-0.3,0.85);
    \node[right] at (1.2,0.85) {$\scriptstyle 2$};
  \end{tikzpicture}
  \quad\quad
  \begin{tikzpicture}
    \node[below] at (0,-2) {$x_i(t)$};
    \draw (0.75,0) ellipse (0.5 and 2);
    \draw (-0.75,2) arc (90:270:0.5 and 2);
    \draw[dashed] (-0.75,2) arc (90:-90:0.5 and 2);
    \draw[-] (0.75,2) -- (-0.75,2);
    \draw[-] (0.75,-2) -- (-0.75,-2);
    \draw[-] (1.25,0) -- (0.25,0);
    \draw[dashed] (-0.25,0) -- (0.21,0);
    \draw[singlearrow] (1.25,0) -- (-0.25,0);
    \node[right] at (1.25,0) {$\scriptstyle 1$};
    \draw[-] (1.2,-0.85) -- (0.3,-0.85);
    \draw[dashed] (-0.3,-0.85) -- (0.3,-0.85);
    \draw[singlearrow] (1.2,-0.85) -- (-0.3,-0.85);
    \node[right] at (1.2,-0.85) {$\scriptstyle n$};
    \draw[-] (1.015,-1.7) -- (0.49,-1.7);
    \draw[dashed] (-0.49,-1.7) -- (0.45,-1.7);
    \draw[singlearrow] (1.015,-1.7) -- (-0.49,-1.7);
    \node[right] at (1.015,-1.7) {$\scriptstyle n-1$};
    \draw[-] (-1.225,-0.65) -- (0.275,-0.65);
    \draw[doublearrow,draw] (0.275,-0.65) -- (-1.225,-0.65);
    \node[right] at (0.275,-0.65) {$\scriptstyle i$};
    \draw[singlearrow,draw] (-0.475,-0.65) arc (199:174.2:0.5 and 2);
    \node at (-0.7,-0.225) {$\scriptstyle t$};
    \draw[-] (-1.245,0.2) -- (0.255,0.2);
    \draw[doublearrow,draw] (0.255,0.2) -- (-1.245,0.2);
    \node[right] at (0.255,0.2) {$\scriptstyle i-1$};
    \draw[-] (1.015,1.7) -- (0.49,1.7);
    \draw[dashed] (-0.49,1.7) -- (0.45,1.7);
    \draw[singlearrow] (1.015,1.7) -- (-0.49,1.7);
    \node[right] at (1.015,1.7) {$\scriptstyle 3$};
    \draw[-] (1.2,0.85) -- (0.3,0.85);
    \draw[dashed] (-0.3,0.85) -- (0.3,0.85);
    \draw[singlearrow] (1.2,0.85) -- (-0.3,0.85);
    \node[right] at (1.2,0.85) {$\scriptstyle 2$};
  \end{tikzpicture}
  \caption{The constituent networks used to define $\cN^{u,v}$.}
  \label{fig:networks on cylinders}
\end{figure}
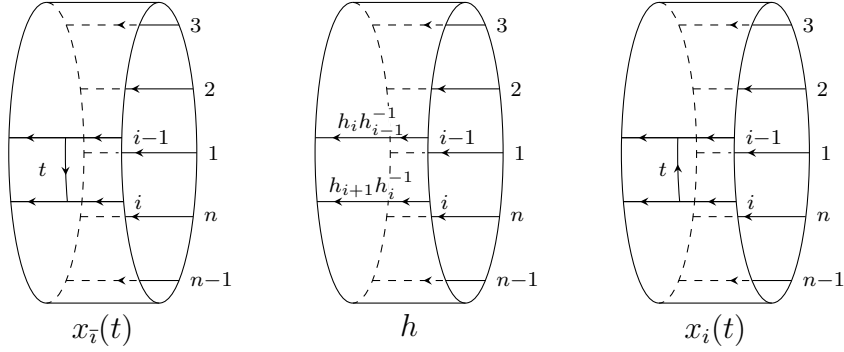

\begin{definition}
  \label{def:network}
  Let $u$, $v$ be elements of the Weyl group of $\widetilde{LSL}_n$ together with a choice of reduced words $u=s_{i_1}\cdots s_{i_{\ell(u)}}$, $v=s_{j_1}\cdots s_{j_{\ell(v)}}$.
  We associate to this data a weighted directed network $\cN^{u,v}$ on a cylinder, by abuse omitting the choice of reduced words from the notation.
  The edge weights of the network will take values in the Laurent polynomial ring
  \[
    \kk[t_{1}^{\pm1},\dotsc,t_{\ell(v)}^{\pm 1},h_1^{\pm 1},\dotsc,h_{n}^{\pm 1},t_{\ol{1}}^{\pm 1},\dotsc,t_{\ol{\ell(u)}}^{\pm 1}].
  \]
  We declare one boundary component of the cylinder to be incoming and the other outgoing, and begin by associating a smaller cylindrical network to each factor in Equation~\eqref{eq:uvfactorization} (see Figure~\ref{fig:networks on cylinders}).

  For $1 \leq i \leq n$, the networks associated to $x_i(t)$ and $x_{\ol{\imath}}(t)$ contain $n$ horizontal \newword{levels} running from the incoming boundary to the outgoing boundary.
  These are labelled by $\ZZ_n$ compatibly with their cyclic ordering and all have weight $1$.
  There is also a vertical \newword{bridge} having weight $t$ running between the $(i-1)$-st level and the $i$-th level; for $x_i(t)$ this bridge is oriented from the $i$-th level to the $(i-1)$-st level, while for $x_{\ol{\imath}}(t)$ this bridge is oriented from the $(i-1)$-st level to the $i$-th.

  The network associated to the Cartan factor $h$ again has $n$ horizontal levels running from the incoming boundary to the outgoing boundary and labelled by $\ZZ_n$.
  Here the $i$-th level is of weight $h_{i+1}h_{i}^{-1}$.

  The network $\cN^{u,v}$ is now defined by gluing these constituent networks together in the order they appear in Equation~\eqref{eq:uvfactorization} (see Figure~\ref{fig:totalnetwork}).
  Given two adjacent factors, the incoming boundary of the left factor is glued to the outgoing boundary of the right factor so that the horizontal levels of each are aligned compatibly with their labeling by $\ZZ_n$.
\end{definition}

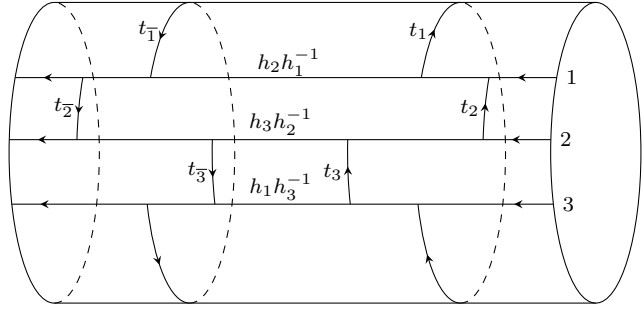
\begin{figure}
	\centering
  \begin{tikzpicture}
    \pgfmathsetmacro{\ehgt}{2} 
    \pgfmathsetmacro{\edpt}{0.6} 
    \pgfmathsetmacro{\ewdt}{0.9} 
    \pgfmathsetmacro{\fang}{150} 
    \pgfmathsetmacro{\ang}{25} 
    \pgfmathsetmacro{\ask}{-7} 
    \pgfmathsetmacro{\hsk}{-0.2} 

    \draw ([shift={(90:{\edpt} and \ehgt)}]0,0) arc (90:270:{\edpt} and \ehgt);
    \draw [dashed] ([shift={(-90:{\edpt} and \ehgt)}]0,0) arc (-90:90:{\edpt} and \ehgt);

    \draw [singlearrow,draw] ([shift={(\fang:{\edpt} and \ehgt)}]\ewdt,0) arc (\fang:\fang+\ang:{\edpt} and \ehgt);
    \node at ([shift={(\fang+0.5*\ang:{\edpt} and \ehgt)}]\ewdt+\hsk,0) {$\scriptstyle t_{\ol{2}}$};

    \draw [singlearrow,draw] ([shift={(90:{\edpt} and \ehgt)}]2*\ewdt,0) arc (90:\fang:{\edpt} and \ehgt);
    \draw [dashed] ([shift={(-90:{\edpt} and \ehgt)}]2*\ewdt,0) arc (-90:90:{\edpt} and \ehgt);
    \draw [singlearrow,draw] ([shift={(\fang+2*\ang:{\edpt} and \ehgt)}]2*\ewdt,0) arc (\fang+2*\ang:270:{\edpt} and \ehgt);
    \node at ([shift={({0.6*(\fang-90)+90}:{\edpt} and \ehgt)}]2*\ewdt+\hsk,0) {$\scriptstyle t_{\ol{1}}$};

    \draw [singlearrow,draw] ([shift={(\fang+\ang:{\edpt} and \ehgt)}]3*\ewdt,0) arc (\fang+\ang:\fang+2*\ang:{\edpt} and \ehgt);
    \node at ([shift={(\fang+1.5*\ang:{\edpt} and \ehgt)}]3*\ewdt+\hsk,0) {$\scriptstyle t_{\ol{3}}$};

    \draw ([shift={(\fang:{\edpt} and \ehgt)}]0,0) -- ([shift={(150:{\edpt} and \ehgt)}]8*\ewdt,0);
    \draw [singlearrowreversed] ([shift={(\fang:{\edpt} and \ehgt)}]0,0) -- ([shift={(150:{\edpt} and \ehgt)}]\ewdt,0);
    \draw [singlearrowreversed] ([shift={(\fang:{\edpt} and \ehgt)}]7*\ewdt,0) -- ([shift={(150:{\edpt} and \ehgt)}]8*\ewdt,0);
    \node at ([shift={(\fang+\ask:{\edpt} and \ehgt)}]4*\ewdt,0) {$\scriptstyle h_2h_1^{-1}$};

    \draw ([shift={(\fang+\ang:{\edpt} and \ehgt)}]0,0) -- ([shift={(\fang+\ang:{\edpt} and \ehgt)}]8*\ewdt,0);
    \draw [singlearrowreversed] ([shift={(\fang+\ang:{\edpt} and \ehgt)}]0,0) -- ([shift={(\fang+\ang:{\edpt} and \ehgt)}]\ewdt,0);
    \draw [singlearrowreversed] ([shift={(\fang+\ang:{\edpt} and \ehgt)}]7*\ewdt,0) -- ([shift={(\fang+\ang:{\edpt} and \ehgt)}]8*\ewdt,0);
    \node at ([shift={(\fang+\ask+\ang:{\edpt} and \ehgt)}]4*\ewdt,0) {$\scriptstyle h_3h_2^{-1}$};

    \draw ([shift={(\fang+2*\ang:{\edpt} and \ehgt)}]0,0) -- ([shift={(\fang+2*\ang:{\edpt} and \ehgt)}]8*\ewdt,0);
    \draw [singlearrowreversed] ([shift={(\fang+2*\ang:{\edpt} and \ehgt)}]0,0) -- ([shift={(\fang+2*\ang:{\edpt} and \ehgt)}]\ewdt,0);
    \draw [singlearrowreversed] ([shift={(\fang+2*\ang:{\edpt} and \ehgt)}]7*\ewdt,0) -- ([shift={(\fang+2*\ang:{\edpt} and \ehgt)}]8*\ewdt,0);
    \node at ([shift={(\fang+\ask+2*\ang:{\edpt} and \ehgt)}]4*\ewdt,0) {$\scriptstyle h_1h_3^{-1}$};

    \draw [singlearrowreversed,draw] ([shift={(\fang+\ang:{\edpt} and \ehgt)}]5*\ewdt,0) arc (\fang+\ang:\fang+2*\ang:{\edpt} and \ehgt);
    \node at ([shift={(\fang+1.5*\ang:{\edpt} and \ehgt)}]5*\ewdt+\hsk,0) {$\scriptstyle t_{3}$};

    \draw [singlearrowreversed,draw] ([shift={(90:{\edpt} and \ehgt)}]6*\ewdt,0) arc (90:\fang:{\edpt} and \ehgt);
    \draw [dashed] ([shift={(-90:{\edpt} and \ehgt)}]6*\ewdt,0) arc (-90:90:{\edpt} and \ehgt);
    \draw [singlearrowreversed,draw] ([shift={(\fang+2*\ang:{\edpt} and \ehgt)}]6*\ewdt,0) arc (\fang+2*\ang:270:{\edpt} and \ehgt);
    \node at ([shift={({0.6*(\fang-90)+90}:{\edpt} and \ehgt)}]6*\ewdt+\hsk,0) {$\scriptstyle t_{1}$};

    \draw [singlearrowreversed,draw] ([shift={(\fang:{\edpt} and \ehgt)}]7*\ewdt,0) arc (\fang:\fang+\ang:{\edpt} and \ehgt);
    \node at ([shift={(\fang+0.5*\ang:{\edpt} and \ehgt)}]7*\ewdt+\hsk,0) {$\scriptstyle t_{2}$};

    \draw (8*\ewdt,0) ellipse ({\edpt} and \ehgt);
    \node at ([shift={(\fang:{\edpt} and \ehgt)}]8*\ewdt-\hsk,0) {$\scriptstyle 1$};
    \node at ([shift={(\fang+\ang:{\edpt} and \ehgt)}]8*\ewdt-\hsk,0) {$\scriptstyle 2$};
    \node at ([shift={(\fang+2*\ang:{\edpt} and \ehgt)}]8*\ewdt-\hsk,0) {$\scriptstyle 3$};

    \draw ([shift={(90:{\edpt} and \ehgt)}]0,0) -- ([shift={(90:{\edpt} and \ehgt)}]8*\ewdt,0);
    \draw ([shift={(-90:{\edpt} and \ehgt)}]0,0) -- ([shift={(-90:{\edpt} and \ehgt)}]8*\ewdt,0);
  \end{tikzpicture}
  \caption{
    The network $\cN^{c,c^{-1}}$ for $c = s_2 s_1 s_3$.
    It encodes the factorization of a generic element of $\widetilde{LSL}_3^{c,c^{-1}}$ as $g = x_{\ol{2}}(t_{\ol{2}})x_{\ol{1}}(t_{\ol{1}})x_{\ol{3}}(t_{\ol{3}})hx_{3}(t_{3})x_{1}(t_{1})x_{2}(t_{2})$.
  }
  \label{fig:totalnetwork}
\end{figure}

By a \newword{path} in $\cN^{u,v}$ we always mean a path which begins at an incoming boundary vertex, is directed compatibly with the edges of $\cN^{u,v}$, and ends at an outgoing boundary vertex.
If $P = \{P_i\}$ is a collection of paths, we write $\partial_{in}P \subset \ZZ_n$ for the set of incoming boundary vertices at which paths in $P$ begin; $\partial_{out}P \subset \ZZ_n$ denotes the corresponding set of outgoing boundary vertices.
The \newword{weight} of a path is the product of the weights of the edges it traverses, and the weight $wt(P)$ of a collection of paths is the product of the weights of its constituent paths.

Given $S \subsetneq \ZZ_n$, let
\[
  \omega_S
  =
  \sum_{i \in S} \omega_{i+1} - \omega_i \in P^\circ \subset \widehat{P}.
\]
Recall that $P^\circ$, the weight lattice of $SL_n$, is contained in the complement of both the Tits cone and its negative in $\widehat{P}$.
The level zero representation $V(\omega_S)$ is $\big(\bigwedge^{|S|}\kk^n\big)[\loopvar^{\pm 1}]$ with its natural $\widetilde{LSL}_n$ action.
Definition~\ref{def:network} is motivated by the following observation:

\begin{proposition}
  \label{prop:minorsfrompaths}
  Suppose $g\in G^{u,v}$ factors as
  \[
  g = x_{\ol{\imath_1}}(t_{\ol{1}})\cdots x_{\ol{\imath_{\ell(u)}}}(t_{\ol{\ell(u)}})hx_{j_1}(t_{1}) \cdots x_{j_{\ell(v)}}(t_{\ell(v)})
  \]
  and let $h_i:= h^{\omega_i}$.
  Then
  \[
    \Delta_{\omega_S}(g) = \sum_{P: S \to S} wt(P),
  \]
  where the sum is over collections $P$ of pairwise disjoint paths in $\cN^{u,v}$ with $\partial_{in}P = \partial_{out}P = S$.
\end{proposition}
\begin{proof}
  The representation $\big(\bigwedge^{|S|}\kk^n\big)[\loopvar^{\pm 1}]$ has a natural basis $\{ e_{k_1} \wedge \cdots \wedge e_{k_{|S|}} \loopvar^d\}$ labelled by pairs of an $|S|$-element subset of $\ZZ_n$ and an integer $d$.
  The networks associated to $x_i(t)$, $x_{\ol{\imath}}(t)$, and $h$ directly encode the action of the corresponding group elements in this basis (see \cite{GSV12,FM14} for related constructions).
  If $S = \{k_1,\dotsc,k_{|S|}\}$, then by a generalized Lindstr\"om-Gessel-Viennot argument, the sum over collections of nonintersecting paths with $\partial_{in}P = \partial_{out}P = S$ computes the diagonal matrix coefficient by which $g$ rescales the basis element $e_{k_1} \wedge \cdots \wedge e_{k_{|S|}}$.
  Since this element has weight $\omega_S$, this matrix coefficient is precisely $\Delta_{\omega_S}(g)$.
\end{proof}

\begin{example}
  \label{ex:minor_computation}
  We use Proposition~\ref{prop:minorsfrompaths} to compute the value of the generalized minor $\Delta_{\omega_1-\omega_2}$ on the generic element $g$ represented by the network of Figure~\ref{fig:totalnetwork}.
  In this case $S=\{2,3\}$ and we get
  \[
    \delimitershortfall -2pt
    \begin{array}{ccccccc}
      \Delta_{\omega_1-\omega_2}(g)\!
      &
      =
      &
      wt\left(
        \begin{tikzpicture}[baseline={([yshift=-0.5ex]current bounding box.center)}]
          \pgfmathsetmacro{\scaling}{0.298} 
          \pgfmathsetmacro{\pathwdt}{5*\scaling} 
          \pgfmathsetmacro{\ehgt}{2*\scaling} 
          \pgfmathsetmacro{\edpt}{0.6*\scaling} 
          \pgfmathsetmacro{\ewdt}{0.9*\scaling} 
          \pgfmathsetmacro{\fang}{150} 
          \pgfmathsetmacro{\ang}{25} 
          \pgfmathsetmacro{\ask}{-7} 
          \pgfmathsetmacro{\hsk}{-0.2*\scaling} 

          \draw ([shift={(90:{\edpt} and \ehgt)}]0,0) arc (90:270:{\edpt} and \ehgt);
          \draw [dashed] ([shift={(-90:{\edpt} and \ehgt)}]0,0) arc (-90:90:{\edpt} and \ehgt);

          \draw ([shift={(\fang:{\edpt} and \ehgt)}]\ewdt,0) arc (\fang:\fang+\ang:{\edpt} and \ehgt);

          \draw ([shift={(90:{\edpt} and \ehgt)}]2*\ewdt,0) arc (90:\fang:{\edpt} and \ehgt);
          \draw [dashed] ([shift={(-90:{\edpt} and \ehgt)}]2*\ewdt,0) arc (-90:90:{\edpt} and \ehgt);
          \draw ([shift={(\fang+2*\ang:{\edpt} and \ehgt)}]2*\ewdt,0) arc (\fang+2*\ang:270:{\edpt} and \ehgt);

          \draw ([shift={(\fang+\ang:{\edpt} and \ehgt)}]3*\ewdt,0) arc (\fang+\ang:\fang+2*\ang:{\edpt} and \ehgt);

          \draw ([shift={(\fang:{\edpt} and \ehgt)}]0,0) -- ([shift={(150:{\edpt} and \ehgt)}]8*\ewdt,0);
          \draw ([shift={(\fang:{\edpt} and \ehgt)}]0,0) -- ([shift={(150:{\edpt} and \ehgt)}]\ewdt,0);
          \draw ([shift={(\fang:{\edpt} and \ehgt)}]7*\ewdt,0) -- ([shift={(150:{\edpt} and \ehgt)}]8*\ewdt,0);

          \draw ([shift={(\fang+\ang:{\edpt} and \ehgt)}]0,0) -- ([shift={(\fang+\ang:{\edpt} and \ehgt)}]8*\ewdt,0);
          \draw ([shift={(\fang+\ang:{\edpt} and \ehgt)}]0,0) -- ([shift={(\fang+\ang:{\edpt} and \ehgt)}]\ewdt,0);
          \draw ([shift={(\fang+\ang:{\edpt} and \ehgt)}]7*\ewdt,0) -- ([shift={(\fang+\ang:{\edpt} and \ehgt)}]8*\ewdt,0);

          \draw ([shift={(\fang+2*\ang:{\edpt} and \ehgt)}]0,0) -- ([shift={(\fang+2*\ang:{\edpt} and \ehgt)}]8*\ewdt,0);
          \draw ([shift={(\fang+2*\ang:{\edpt} and \ehgt)}]0,0) -- ([shift={(\fang+2*\ang:{\edpt} and \ehgt)}]\ewdt,0);
          \draw ([shift={(\fang+2*\ang:{\edpt} and \ehgt)}]7*\ewdt,0) -- ([shift={(\fang+2*\ang:{\edpt} and \ehgt)}]8*\ewdt,0);

          \draw ([shift={(\fang+\ang:{\edpt} and \ehgt)}]5*\ewdt,0) arc (\fang+\ang:\fang+2*\ang:{\edpt} and \ehgt);

          \draw ([shift={(90:{\edpt} and \ehgt)}]6*\ewdt,0) arc (90:\fang:{\edpt} and \ehgt);
          \draw [dashed] ([shift={(-90:{\edpt} and \ehgt)}]6*\ewdt,0) arc (-90:90:{\edpt} and \ehgt);
          \draw ([shift={(\fang+2*\ang:{\edpt} and \ehgt)}]6*\ewdt,0) arc (\fang+2*\ang:270:{\edpt} and \ehgt);

          \draw ([shift={(\fang:{\edpt} and \ehgt)}]7*\ewdt,0) arc (\fang:\fang+\ang:{\edpt} and \ehgt);

          \draw (8*\ewdt,0) ellipse ({\edpt} and \ehgt);

          \draw ([shift={(90:{\edpt} and \ehgt)}]0,0) -- ([shift={(90:{\edpt} and \ehgt)}]8*\ewdt,0);
          \draw ([shift={(-90:{\edpt} and \ehgt)}]0,0) -- ([shift={(-90:{\edpt} and \ehgt)}]8*\ewdt,0);

          \draw [stealth-, line width=\pathwdt] ([shift={(\fang+\ang:{\edpt} and \ehgt)}]0,0) -- ([shift={(\fang+\ang:{\edpt} and \ehgt)}]8*\ewdt,0);
          \draw [stealth-, line width=\pathwdt] ([shift={(\fang+2*\ang:{\edpt} and \ehgt)}]0,0) -- ([shift={(\fang+2*\ang:{\edpt} and \ehgt)}]8*\ewdt,0);

        \end{tikzpicture}
      \right)
      &
      +
      &
      wt\left(
        \begin{tikzpicture}[baseline={([yshift=-0.5ex]current bounding box.center)}]
          \pgfmathsetmacro{\scaling}{0.298} 
          \pgfmathsetmacro{\pathwdt}{5*\scaling} 
          \pgfmathsetmacro{\ehgt}{2*\scaling} 
          \pgfmathsetmacro{\edpt}{0.6*\scaling} 
          \pgfmathsetmacro{\ewdt}{0.9*\scaling} 
          \pgfmathsetmacro{\fang}{150} 
          \pgfmathsetmacro{\ang}{25} 
          \pgfmathsetmacro{\ask}{-7} 
          \pgfmathsetmacro{\hsk}{-0.2*\scaling} 

          \draw ([shift={(90:{\edpt} and \ehgt)}]0,0) arc (90:270:{\edpt} and \ehgt);
          \draw [dashed] ([shift={(-90:{\edpt} and \ehgt)}]0,0) arc (-90:90:{\edpt} and \ehgt);

          \draw ([shift={(\fang:{\edpt} and \ehgt)}]\ewdt,0) arc (\fang:\fang+\ang:{\edpt} and \ehgt);

          \draw ([shift={(90:{\edpt} and \ehgt)}]2*\ewdt,0) arc (90:\fang:{\edpt} and \ehgt);
          \draw [dashed] ([shift={(-90:{\edpt} and \ehgt)}]2*\ewdt,0) arc (-90:90:{\edpt} and \ehgt);
          \draw ([shift={(\fang+2*\ang:{\edpt} and \ehgt)}]2*\ewdt,0) arc (\fang+2*\ang:270:{\edpt} and \ehgt);

          \draw ([shift={(\fang+\ang:{\edpt} and \ehgt)}]3*\ewdt,0) arc (\fang+\ang:\fang+2*\ang:{\edpt} and \ehgt);

          \draw ([shift={(\fang:{\edpt} and \ehgt)}]0,0) -- ([shift={(150:{\edpt} and \ehgt)}]8*\ewdt,0);
          \draw ([shift={(\fang:{\edpt} and \ehgt)}]0,0) -- ([shift={(150:{\edpt} and \ehgt)}]\ewdt,0);
          \draw ([shift={(\fang:{\edpt} and \ehgt)}]7*\ewdt,0) -- ([shift={(150:{\edpt} and \ehgt)}]8*\ewdt,0);

          \draw ([shift={(\fang+\ang:{\edpt} and \ehgt)}]0,0) -- ([shift={(\fang+\ang:{\edpt} and \ehgt)}]8*\ewdt,0);
          \draw ([shift={(\fang+\ang:{\edpt} and \ehgt)}]0,0) -- ([shift={(\fang+\ang:{\edpt} and \ehgt)}]\ewdt,0);
          \draw ([shift={(\fang+\ang:{\edpt} and \ehgt)}]7*\ewdt,0) -- ([shift={(\fang+\ang:{\edpt} and \ehgt)}]8*\ewdt,0);

          \draw ([shift={(\fang+2*\ang:{\edpt} and \ehgt)}]0,0) -- ([shift={(\fang+2*\ang:{\edpt} and \ehgt)}]8*\ewdt,0);
          \draw ([shift={(\fang+2*\ang:{\edpt} and \ehgt)}]0,0) -- ([shift={(\fang+2*\ang:{\edpt} and \ehgt)}]\ewdt,0);
          \draw ([shift={(\fang+2*\ang:{\edpt} and \ehgt)}]7*\ewdt,0) -- ([shift={(\fang+2*\ang:{\edpt} and \ehgt)}]8*\ewdt,0);

          \draw ([shift={(\fang+\ang:{\edpt} and \ehgt)}]5*\ewdt,0) arc (\fang+\ang:\fang+2*\ang:{\edpt} and \ehgt);

          \draw ([shift={(90:{\edpt} and \ehgt)}]6*\ewdt,0) arc (90:\fang:{\edpt} and \ehgt);
          \draw [dashed] ([shift={(-90:{\edpt} and \ehgt)}]6*\ewdt,0) arc (-90:90:{\edpt} and \ehgt);
          \draw ([shift={(\fang+2*\ang:{\edpt} and \ehgt)}]6*\ewdt,0) arc (\fang+2*\ang:270:{\edpt} and \ehgt);

          \draw ([shift={(\fang:{\edpt} and \ehgt)}]7*\ewdt,0) arc (\fang:\fang+\ang:{\edpt} and \ehgt);

          \draw (8*\ewdt,0) ellipse ({\edpt} and \ehgt);

          \draw ([shift={(90:{\edpt} and \ehgt)}]0,0) -- ([shift={(90:{\edpt} and \ehgt)}]8*\ewdt,0);
          \draw ([shift={(-90:{\edpt} and \ehgt)}]0,0) -- ([shift={(-90:{\edpt} and \ehgt)}]8*\ewdt,0);

          \draw [stealth-, line width=\pathwdt] ([shift={(\fang+\ang:{\edpt} and \ehgt)}]0,0) -- ([shift={(\fang+\ang:{\edpt} and \ehgt)}]\ewdt,0) arc (\fang+\ang:\fang:{\edpt} and \ehgt) -- ([shift={(\fang:{\edpt} and \ehgt)}]7*\ewdt,0) arc (\fang:\fang+\ang:{\edpt} and \ehgt) -- ([shift={(\fang+\ang:{\edpt} and \ehgt)}]8*\ewdt,0);
          \draw [stealth-, line width=\pathwdt] ([shift={(\fang+2*\ang:{\edpt} and \ehgt)}]0,0) -- ([shift={(\fang+2*\ang:{\edpt} and \ehgt)}]8*\ewdt,0);

        \end{tikzpicture}
      \right)
      &
      +
      &
      wt\left(
        \begin{tikzpicture}[baseline={([yshift=-0.5ex]current bounding box.center)}]
          \label{fig:level_zero_example}
          \pgfmathsetmacro{\scaling}{0.298} 
          \pgfmathsetmacro{\pathwdt}{5*\scaling} 
          \pgfmathsetmacro{\ehgt}{2*\scaling} 
          \pgfmathsetmacro{\edpt}{0.6*\scaling} 
          \pgfmathsetmacro{\ewdt}{0.9*\scaling} 
          \pgfmathsetmacro{\fang}{150} 
          \pgfmathsetmacro{\ang}{25} 
          \pgfmathsetmacro{\ask}{-7} 
          \pgfmathsetmacro{\hsk}{-0.2*\scaling} 

          \draw ([shift={(90:{\edpt} and \ehgt)}]0,0) arc (90:270:{\edpt} and \ehgt);
          \draw [dashed] ([shift={(-90:{\edpt} and \ehgt)}]0,0) arc (-90:90:{\edpt} and \ehgt);

          \draw ([shift={(\fang:{\edpt} and \ehgt)}]\ewdt,0) arc (\fang:\fang+\ang:{\edpt} and \ehgt);

          \draw ([shift={(90:{\edpt} and \ehgt)}]2*\ewdt,0) arc (90:\fang:{\edpt} and \ehgt);
          \draw [dashed] ([shift={(-90:{\edpt} and \ehgt)}]2*\ewdt,0) arc (-90:90:{\edpt} and \ehgt);
          \draw ([shift={(\fang+2*\ang:{\edpt} and \ehgt)}]2*\ewdt,0) arc (\fang+2*\ang:270:{\edpt} and \ehgt);

          \draw ([shift={(\fang+\ang:{\edpt} and \ehgt)}]3*\ewdt,0) arc (\fang+\ang:\fang+2*\ang:{\edpt} and \ehgt);

          \draw ([shift={(\fang:{\edpt} and \ehgt)}]0,0) -- ([shift={(150:{\edpt} and \ehgt)}]8*\ewdt,0);
          \draw ([shift={(\fang:{\edpt} and \ehgt)}]0,0) -- ([shift={(150:{\edpt} and \ehgt)}]\ewdt,0);
          \draw ([shift={(\fang:{\edpt} and \ehgt)}]7*\ewdt,0) -- ([shift={(150:{\edpt} and \ehgt)}]8*\ewdt,0);

          \draw ([shift={(\fang+\ang:{\edpt} and \ehgt)}]0,0) -- ([shift={(\fang+\ang:{\edpt} and \ehgt)}]8*\ewdt,0);
          \draw ([shift={(\fang+\ang:{\edpt} and \ehgt)}]0,0) -- ([shift={(\fang+\ang:{\edpt} and \ehgt)}]\ewdt,0);
          \draw ([shift={(\fang+\ang:{\edpt} and \ehgt)}]7*\ewdt,0) -- ([shift={(\fang+\ang:{\edpt} and \ehgt)}]8*\ewdt,0);

          \draw ([shift={(\fang+2*\ang:{\edpt} and \ehgt)}]0,0) -- ([shift={(\fang+2*\ang:{\edpt} and \ehgt)}]8*\ewdt,0);
          \draw ([shift={(\fang+2*\ang:{\edpt} and \ehgt)}]0,0) -- ([shift={(\fang+2*\ang:{\edpt} and \ehgt)}]\ewdt,0);
          \draw ([shift={(\fang+2*\ang:{\edpt} and \ehgt)}]7*\ewdt,0) -- ([shift={(\fang+2*\ang:{\edpt} and \ehgt)}]8*\ewdt,0);

          \draw ([shift={(\fang+\ang:{\edpt} and \ehgt)}]5*\ewdt,0) arc (\fang+\ang:\fang+2*\ang:{\edpt} and \ehgt);

          \draw ([shift={(90:{\edpt} and \ehgt)}]6*\ewdt,0) arc (90:\fang:{\edpt} and \ehgt);
          \draw [dashed] ([shift={(-90:{\edpt} and \ehgt)}]6*\ewdt,0) arc (-90:90:{\edpt} and \ehgt);
          \draw ([shift={(\fang+2*\ang:{\edpt} and \ehgt)}]6*\ewdt,0) arc (\fang+2*\ang:270:{\edpt} and \ehgt);

          \draw ([shift={(\fang:{\edpt} and \ehgt)}]7*\ewdt,0) arc (\fang:\fang+\ang:{\edpt} and \ehgt);

          \draw (8*\ewdt,0) ellipse ({\edpt} and \ehgt);

          \draw ([shift={(90:{\edpt} and \ehgt)}]0,0) -- ([shift={(90:{\edpt} and \ehgt)}]8*\ewdt,0);
          \draw ([shift={(-90:{\edpt} and \ehgt)}]0,0) -- ([shift={(-90:{\edpt} and \ehgt)}]8*\ewdt,0);

          \draw [stealth-, line width=\pathwdt] ([shift={(\fang+\ang:{\edpt} and \ehgt)}]0,0) -- ([shift={(\fang+\ang:{\edpt} and \ehgt)}]\ewdt,0) arc (\fang+\ang:\fang:{\edpt} and \ehgt) -- ([shift={(\fang:{\edpt} and \ehgt)}]7*\ewdt,0) arc (\fang:\fang+\ang:{\edpt} and \ehgt) -- ([shift={(\fang+\ang:{\edpt} and \ehgt)}]8*\ewdt,0);
          \draw [stealth-, line width=\pathwdt] ([shift={(\fang+2*\ang:{\edpt} and \ehgt)}]0,0) -- ([shift={(\fang+2*\ang:{\edpt} and \ehgt)}]3*\ewdt,0) arc (\fang+2*\ang:\fang+\ang:{\edpt} and \ehgt) -- ([shift={(\fang+\ang:{\edpt} and \ehgt)}]5*\ewdt,0) arc (\fang+\ang:\fang+2*\ang:{\edpt} and \ehgt) -- ([shift={(\fang+2*\ang:{\edpt} and \ehgt)}]8*\ewdt,0);

        \end{tikzpicture}
      \right)
      \\
      \\
      &
      =
      &
      h_3h_2^{-1}
      \cdot
      h_1h_3^{-1}
      &
      +
      &
      t_{\ol{2}}h_2h_1^{-1}t_2
      \cdot
      h_1h_3^{-1}
      &
      +
      &
      t_{\ol{2}}h_2h_1^{-1}t_2
      \cdot
      t_{\ol{3}}h_3h_2^{-1}t_3
      \\
      \\
      &
      =
      &
      \multicolumn{5}{l}{h_1h_2^{-1} + t_{\ol{2}}t_2h_2h_3^{-1} + t_{\ol{2}}t_2t_{\ol{3}}t_3h_1^{-1}h_3}
    \end{array}
  \]
\end{example}

We now return to the specific case of $\widetilde{LSL}_n^{c,c^{-1}}$, where $c = s_{\sigma_1}\cdots s_{\sigma_{n}}$ is a fixed Coxeter element.
From now on we reindex the factorization of Equation~\eqref{eq:uvfactorization} in a way tailored to this case:
\begin{equation}
  \label{eq:coxfactorization}
  g = x_{\ol{\sigma_1}}(t_{\ol{\sigma_1}}) \cdots x_{\ol{\sigma_n}}(t_{\ol{\sigma_n}})h x_{\sigma_n}(t_{\sigma_n}) \cdots x_{\sigma_1}(t_{\sigma_1}).
\end{equation}
To use the preceding proposition effectively, we collect some elementary combinatorial observations about collections of pairwise disjoint paths $P$ in $\cN^{c,c^{-1}}$.
Note that when $u$ and $v$ are Coxeter elements the choice of reduced word in Definition~\ref{def:network} is immaterial: different reduced words differ by relations of the form $s_i s_j = s_j s_i$ for $a_{ij} = 0$, and these relations do not change the combinatorics of the network.
Given a collection $P$ of disjoint paths, define its \newword{bridge set} $\beta(P)\subset\ZZ_n$ by letting $j\in\beta(P)$ if and only if a path in $P$ traverses the unique bridge of weight $t_j$.

\begin{lemma}
  \label{lem:phi}
  Let $P$ be a collection of pairwise disjoint paths in $\cN^{c,c^{-1}}$.
  Then:
  \begin{enumerate}
    \item for $j\in\partial_{in}P$, $j+1\in\beta(P)$ implies $j+1\to j$ in $Q$ and $j\in\beta(P)$;
    \item for $j\notin\partial_{in}P$, $j\in\beta(P)$ implies $j\to j+1$ in $Q$ and $j+1\in\beta(P)$.
  \end{enumerate}
\end{lemma}
\begin{proof}
  If $c = s_{\sigma_1}\cdots s_{\sigma_{n}}$, then by our definitions $(\sigma_1,\ldots,\sigma_{n})$ is sink adapted for $Q$.
  The orientation of the arrow connecting vertices $j$ and $j+1$ is thus related to the local structure of the network $\cN^{c,c^{-1}}$ as in Figure~\ref{fig:quivernetwork}.
  The result follows immediately by comparing the statements against the picture.
\end{proof}

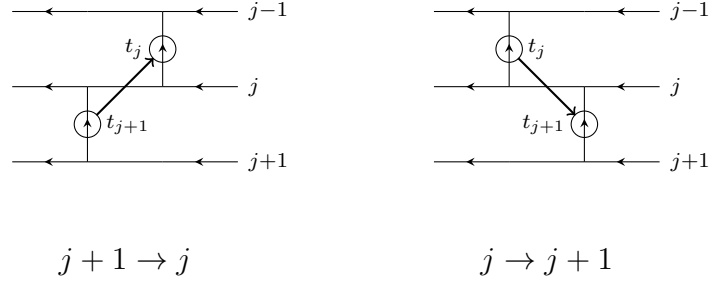
\begin{figure}
	\centering
  \begin{tikzpicture}
    \node[below] at (0,-2) {$j+1\to j$};
    \draw[singlearrow,draw] (1.5,1) -- (0.5,1);
    \draw[-] (0.5,1) -- (-0.5,1);
    \draw[singlearrow,draw] (-0.5,1) -- (-1.5,1);
    \node[right] at (1.5,1) {$\scriptstyle j-1$};
    \draw[singlearrow,draw] (0.5,0) -- (0.5,1);
    \node[left] at (0.5,0.5) {$\scriptstyle t_j\ $};
    \draw[singlearrow,draw] (1.5,0) -- (0.5,0);
    \draw[-] (0.5,0) -- (-0.5,0);
    \draw[singlearrow,draw] (-0.5,0) -- (-1.5,0);
    \node[right] at (1.5,0) {$\scriptstyle j$};
    \draw[singlearrow,draw] (-0.5,-1) -- (-0.5,0);
    \node[right] at (-0.5,-0.5) {$\ \scriptstyle t_{j+1}$};
    \draw[singlearrow,draw] (1.5,-1) -- (0.5,-1);
    \draw[-] (0.5,-1) -- (-0.5,-1);
    \draw[singlearrow,draw] (-0.5,-1) -- (-1.5,-1);
    \node[right] at (1.5,-1) {$\scriptstyle j+1$};
    \draw (-0.5,-0.5) circle (5pt);
    \draw (0.5,0.5) circle (5pt);
    \draw[thick,->] (-0.38,-0.38) -- (0.38,0.38);
  \end{tikzpicture}
  \qquad\qquad
  \begin{tikzpicture}
    \node[below] at (0,-2) {$j\to j+1$};
    \draw[singlearrow,draw] (1.5,1) -- (0.5,1);
    \draw[-] (0.5,1) -- (-0.5,1);
    \draw[singlearrow,draw] (-0.5,1) -- (-1.5,1);
    \node[right] at (1.5,1) {$\scriptstyle j-1$};
    \draw[singlearrow,draw] (-0.5,0) -- (-0.5,1);
    \node[right] at (-0.5,0.5) {$\ \scriptstyle t_j$};
    \draw[singlearrow,draw] (1.5,0) -- (0.5,0);
    \draw[-] (0.5,0) -- (-0.5,0);
    \draw[singlearrow,draw] (-0.5,0) -- (-1.5,0);
    \node[right] at (1.5,0) {$\scriptstyle j$};
    \draw[singlearrow,draw] (0.5,-1) -- (0.5,0);
    \node[left] at (0.5,-0.5) {$\scriptstyle t_{j+1}\ $};
    \draw[singlearrow,draw] (1.5,-1) -- (0.5,-1);
    \draw[-] (0.5,-1) -- (-0.5,-1);
    \draw[singlearrow,draw] (-0.5,-1) -- (-1.5,-1);
    \node[right] at (1.5,-1) {$\scriptstyle j+1$};
    \draw (-0.5,0.5) circle (5pt);
    \draw (0.5,-0.5) circle (5pt);
    \draw[thick,->] (-0.38,0.38) -- (0.38,-0.38);
  \end{tikzpicture}
  \caption{The local picture relating the network $\cN^{c,c^{-1}}$ to the orientation of the arrow of $Q$ connecting vertices $j$ and $j+1$, where we draw the vertices on top of the corresponding bridges in $\cN^{c,c^{-1}}$}
  \label{fig:quivernetwork}
\end{figure}

We now turn to the main result of the section.
We will freely use the isomorphism $\cA_{\Qdp} \cong \kk[\widetilde{LSL}_n^{c,c^{-1}}]$ to identify elements of $\cA_{\Qdp}$ with elements of
\[
  \kk[t_{1}^{\pm1},\dotsc,t_{n}^{\pm 1},h_1^{\pm 1},\dotsc,h_{n}^{\pm 1},t_{\ol{1}}^{\pm 1},\dotsc,t_{\ol{n}}^{\pm 1}]
\]
by evaluating them on a generic $g \in \widetilde{LSL}_n^{c,c^{-1}}$ factored as in Equation~\eqref{eq:coxfactorization} (and setting $h_i := h^{\omega_i}$).
For example, we have
\begin{equation}
  \label{eq:typeAyhat}
  x_i = h_i,\quad\quad \hat{y}_i = t_it_{\ol{\imath}}h_{i-1}^{-1}h_i^2h_{i+1}^{-1},
\end{equation}
the latter following by reindexing vertices of $Q$ in the results of Lemma~\ref{lemma:coefficients_values} with the present conventions, then plugging them into the definition of $\hat{y}_i$.

\begin{theorem}
\label{th:cluster character equals minor}
  Let $\Qrep_{[k,\ell]}$ be a rigid regular representation of $Q$.
  Then under $\cA_{\Qdp} \cong \kk[\widetilde{LSL}_n^{c,c^{-1}}]$, the regular cluster variable $x_{\Qrep_{[k,\ell]}}$ is the restriction of the level zero minor $\Delta_{\grep(\Qrep_{[k,\ell]})}$.
\end{theorem}
\begin{proof}
  Recall from Lemma~\ref{lem:regulargvectors} the subset $S_{[k,\ell]} \subset \ZZ_n$ such that $\grep(\Qrep_{[k,\ell]}) = \omega_{S_{[k,\ell]}}$.
  Let $\cP_{[k,\ell]}$ denote the set of pairwise disjoint collections of paths in $\cN^{c,c^{-1}}$ with $\partial_{in}P = \partial_{out}P = S_{[k,\ell]}$.
  We claim that for $P \in \cP_{[k,\ell]}$:
  \begin{enumerate}[\quad\upshape (a)]
    \item
      the bridge set $\beta(P)$ is a target-closed subset of $[k,\ell]$,
    \item
      we have
      \[
        wt(P) = x^{\grep(\Qrep_{[k,\ell]})} \prod_{j \in \beta(P)} \hat{y}_j,
      \]
    \item
      the assignment $P \mapsto \beta(P)$ defines a bijection between $\cP_{[k,\ell]}$ and the set $\cE_{[k,\ell]}$ of target-closed subsets of $[k,\ell]$.
  \end{enumerate}
  The main claim follows from these given Propositions~\ref{prop:regular coindices} and~\ref{prop:minorsfrompaths}.

  As in Proposition~\ref{prop:regular coindices}, we let $i_1 < \cdots < i_m$ and $o_1 < \cdots < o_m$ denote the labels of the sinks and sources of $Q$ contained in $[k,\ell]$.
  We consider the case where the first source in $[k,\ell]$ comes before the first sink (i.e. $o_1 < i_1$ in $[k,\ell]$), the opposite case following by a symmetric argument.
  Recall that in this case
  \[
    S_{[k,\ell]} = [k,o_1-1] \cup [i_1,o_2-1] \cup \cdots \cup [i_m,\ell].
  \]

  Let $P \in \cP_{[k,\ell]}$.
  Since $o_1<i_1$, we have $j+1\to j$ in $Q$ for each $j\in[k^-,k-1]$, where $k^-$ is the nearest sink less that $k$ (i.e. no element of $[k^-+1,k]$ is a sink).
  But $[k^-,k-1]\cap\partial_{in}P=\emptyset$ and thus part (2) of Lemma~\ref{lem:phi} implies $j\notin\beta(P)$ for any $j\in[k^-,k-1]$.
  Similarly, $o_m<i_m$ gives $j\to j-1$ in $Q$ for each $j\in[\ell+1,\ell^+]$, where $\ell^+$ is the nearest source after $\ell$.
  Applying part (2) of Lemma~\ref{lem:phi} inductively beginning with the source $j=\ell^+$, we see that $j\notin\beta(P)$ for any $j\in[\ell+1,\ell^+]$.
  Noting that $[\ell^+,k^-]\cap\partial_{in}P=\emptyset$ and then arguing as above we conclude that $\beta(P)\subset[k,\ell]$.

  Write $i_0=k$ and $i_{m+1}=o_{m+1}=\ell+1$.
  For $1\le r\le m+1$ and $j\in[i_{r-1},o_r-1]\subset\partial_{in}P$, we have $j+1\to j$ in $Q$ and part (1) of Lemma~\ref{lem:phi} says $j+1\in\beta(P)$ implies $j\in\beta(P)$.
  For $1\le r\le m$ and $j\in[o_r,i_r-1]$, we have $j\notin\partial_{in}P$ and $j\to j+1$ in $Q$ so that part (2) of Lemma~\ref{lem:phi} says $j\in\beta(P)$ implies $j+1\in\beta(P)$.
  Putting these observations together, we conclude that $\beta(P)\subset[k,\ell]$ is target-closed and (a) is established.

  Now suppose $E\in\cE_{[k,\ell]}$ is a target-closed subset of $[k,\ell]$.
  Write $P_0\in\cP_{[k,\ell]}$ for the trivial collection of pairwise disjoint paths in which the path beginning on level $j\in S_{[k,\ell]}=\partial_{in}P_0$ stays on level $j$ without traversing any bridges.
  Note that
  \[
    wt(P_0)=\prod_{i \in S_{[k,\ell]}}h_{i+1} h_i^{-1}=x^{\grep(\Qrep_{[k,\ell]})}.
  \]
  We will show how to inductively modify $P_0$ to obtain a unique collection of paths $\pi(E)\in\cP_{[k,\ell]}$ with $\beta(\pi(E))=E$, keeping track of how the weight of the collection changes along the way.
  That is, for $1\le r\le m+1$ we build from $P_{r-1}$ a new collection $P_r \in \cP_{[k,\ell]}$ such that
  \[
    wt(P_r) = x^{\grep(\Qrep_{[k,\ell]})} \prod_{j \in E \cap [k,i_r]} \hat{y}_j.
  \]
  We keep the convention that $i_0=k$ and $i_{m+1}=o_{m+1}=\ell+1$, and define $P_r$ in two steps:
  \begin{enumerate}
    \item
      First define an intermediate network $P'_r \in \cP_{[k,\ell]}$.
      If $i_{r-1} \notin E$, let $P'_r=P_{r-1}$.
      Otherwise let $o'_r$ be the largest index in $[i_{r-1},o_r-1]$ such that $[i_{r-1},o'_r]\subset E$.
      We obtain $P'_r$ from $P_{r-1}$ by, for each $j\in[i_{r-1},o'_r]$, diverting the path that stays on level $j$ so that it traverses the bridges of weights $t_j$ and $t_{\ol{\jmath}}$.
      By inspection of Figure~\ref{fig:quivernetwork} and their counterparts with $t_{\ol{\jmath}}$, the resulting collection $P'_r$ will remain pairwise disjoint.

      If $P'_r \neq P_{r-1}$, then using Equation~\eqref{eq:typeAyhat} we have
      \[
        wt(P'_r)=wt(P_{r-1})\cdot h_{i_{r-1}-1}^{-1}h_{i_{r-1}}h_{o'_r}h_{o'_r+1}^{-1}\prod_{j\in[i_{r-1},o'_r]}(t_jt_{\ol{\jmath}}) = wt(P_{r-1})\prod_{j \in [i_{r-1},o'_r]} \hat{y}_j.
      \]
      In either case, using our inductive assumption on $wt(P_{r-1})$ and the fact that $E$ is target-closed we have
      \[
        wt(P'_r) = x^{\grep(\Qrep_{[k,\ell]})} \prod_{j \in E \cap [k,o_{r}-1]} \hat{y}_j.
      \]

    \item
      If $i_r \notin E$, let $P_r = P'_r$.
      Otherwise let $o''_r$ be the smallest index in $[o_r,i_r]$ such that $[o''_r,i_r]\subset E$.
      We obtain $P_r$ from $P'_r$ by diverting the path that stays on level $i_r$ so that it traverses every bridge of weight $t_j$ and $t_{\ol{\jmath}}$ for $j\in[o''_r,i_r]$.
      By inspection of Figure~\ref{fig:quivernetwork} and their counterparts with $t_{\ol{\jmath}}$, the resulting collection $P_r$ will remain pairwise disjoint since $[o''_r,i_r]\cap\partial_{in}P'_r=\{i_r\}$ and when $o''_r=o_r$ (i.e. $o_r\in E$) we have $o_r-1\in\beta(P'_r)$.

      If $P_r \neq P'_r$, we have
      \[
        wt(P_r)=wt(P'_r)\cdot h_{o''_r-1}^{-1}h_{o''_r}h_{i_r}h_{i_r+1}^{-1}\prod_{j\in[o''_r,i_r]}(t_jt_{\ol{\jmath}})=wt(P'_r)\prod_{j \in [o''_r,i_r]} \hat{y}_j.
      \]
      In either case, again by induction and the fact that $E$ is target-closed we have
      \[
        wt(P_r) = x^{\grep(\Qrep_{[k,\ell]})} \prod_{j \in E \cap [k,i_r]} \hat{y}_j.
      \]
  \end{enumerate}
  By induction, $\pi(E)=P_{m+1} \in \cP_{[k,\ell]}$ has $\beta(\pi(E))=E$ and
  \[
    wt(\pi(E))=x^{\grep(\Qrep_{[k,\ell]})}\prod_{j\in E}\hat y_j,
  \]
  as desired in (b).
  Finally, it is straightforward to see that if $\beta(P) = \beta(P')$ for some $P, P' \in \cP_{[k,\ell]}$, we must have $P = P'$.
  In particular, $\pi(\beta(P))=P$ for any $P\in\cP_{[k,\ell]}$, which shows (c) and completes the proof.
\end{proof}
\begin{example}
  We continue with Example~\ref{ex:minor_computation}, which concerns type $A_2^{(1)}$ with Coxeter element $c=s_2s_1s_3$.
  The regular cluster variable with $\bfg$-vector $\omega_1-\omega_2$ is obtained from the initial cluster of $\cA_{\Bdp}$ by mutating first in direction 3 then in direction 2, and the result is
  \[
    x_{\omega_1-\omega_2}
    =
    \frac{x_1 x_3 + \cvar_2 \cvar_{\ol{2}} + x_1 x_2 \cvar_2 \cvar_3 \cvar_{\ol{2}} \cvar_{\ol{3}}}{x_2 x_3}.
  \]
  If we substitute the values computed in Lemma~\ref{lemma:coefficients_values} into this expression we recover the value of the generalized minor $\Delta_{\omega_1-\omega_2}$ on a generic element $g$ of $\widetilde{LSL}_3^{c,c^{-1}}$ as computed in Example~\ref{ex:minor_computation}.
\end{example}
\subsection{Level zero minors in other affine types}
\label{sec:othertypes}

We anticipate that the realization of regular cluster variables as level zero minors is not specific to type $A_{n-1}^{\!(1)}$.
Indeed, that regular $\bfg$-vectors in any affine type lie on the boundary of the Tits cone and its negative is strongly suggested by both the combinatorial construction of affine generalized associahedra \cite{RS16} and by Cambrian frameworks \cite{RS15}.
With these considerations in hand we pose the following specific form of Conjecture~\ref{conj:mainconjecture} in affine type:

\begin{conjecture}
  \label{conj:affine}
  Let $B$ be an acyclic skew-symmetrizable matrix whose Cartan companion $A$ is of affine type.
  Let $x_{i;t} \in \cA_{\Bdp}$ be any regular cluster variable and $\gv_{i;t} \in P$ its $\bfg$-vector.
  Then under $\cA_{\Bdp} \cong \kk[G^{c,c^{-1}}]$, the cluster variable $x_{i;t}$ is the restriction of the level zero minor $\Delta_{\gv_{i;t}}$.
\end{conjecture}

While the combinatorics we used to prove Conjecture~\ref{conj:affine} in type $A_{n-1}^{\!(1)}$ is not available in general, in some types it is possible to compute the relevant minors by hand using only the characters of the representations involved (e.g. when they are spaces of Laurent polynomials valued in a representation all of whose weight spaces are one dimensional).
The results are consistent with Conjecture~\ref{conj:affine}:

\begin{proposition}
  \label{prop:conjholds}
  Conjecture~\ref{conj:affine} holds when the Cartan companion of $B$ is of type $B_3^{(1)}$, $C_2^{(1)}$, $D_4^{(1)}$, or $G_2^{(1)}$.
\end{proposition}
\begin{proof}
  We illustrate the needed computation in the case of a sample regular cluster variable in type $C_{2}^{(1)}$.
  The full proposition amounts to an assertion about several dozen such calculations, which can be efficiently checked by a computer.

  Consider the exchange matrix and Cartan companion
  \[
    B =
    \begin{bmatrix}
      0 & 1 & 0 \\
      -2 & 0 & 2 \\
      0 & -1 & 0
    \end{bmatrix},
    \quad
    A =
    \begin{bmatrix}
      2 & -1 & 0 \\
      -2 & 2 & -2 \\
      0 & -1 & 2
    \end{bmatrix}.
  \]
  This corresponds to the Coxeter element $s_1 s_2 s_3$, where the $C_2^{(1)}$ Dynkin diagram is labelled as
  
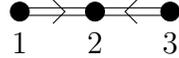
\vspace{3mm}
  \[
    \begin{tikzpicture}[scale=0.5]
      \draw (0, 0.1 cm) -- +(2 cm,0);
      \draw (0, -0.1 cm) -- +(2 cm,0);
      \draw[shift={(1.2, 0)}, rotate=0] (135 : 0.45cm) -- (0,0) -- (-135 : 0.45cm);
      {
        \pgftransformxshift{2 cm}
        \draw (0 cm,0) -- (0 cm,0);
        \draw (0 cm, 0.1 cm) -- +(2 cm,0);
        \draw (0 cm, -0.1 cm) -- +(2 cm,0);
        \draw[shift={(0.8, 0)}, rotate=180] (135 : 0.45cm) -- (0,0) -- (-135 : 0.45cm);
        \draw[fill=black] (0 cm, 0 cm) circle (.25cm) node[below=4pt]{$2$};
        \draw[fill=black] (2 cm, 0 cm) circle (.25cm) node[below=4pt]{$3$};
      }
      \draw[fill=black] (0 cm, 0 cm) circle (.25cm) node[below=4pt]{$1$};
    \end{tikzpicture}
  \]
  If we mutate in directions 1, 3, and 2 in that order we obtain the regular cluster variable
  \[
    x_{\omega_2 - \omega_1} = \frac{x_1 x_2^2 z_1 z_2 z_3 z_{\ol{1}} z_{\ol{2}} z_{\ol{3}} + x_1 z_1 z_2 z_{\ol{1}} z_{\ol{2}} + x_2^2 x_3 +  x_3 z_1 z_{\ol{1}}}{x_1 x_2 x_3}
  \]
  with $\bfg$-vector is $\omega_2 - \omega_1$.

  Let us evaluate the level zero minor $\Delta_{\omega_2 - \omega_1}$ on a generic element
  \[
    g = x_{\ol{1}}(t_{\ol{1}})x_{\ol{2}}(t_{\ol{2}})x_{\ol{3}}(t_{\ol{3}})hx_3(t_3)x_2(t_2)x_1(t_1) \in \widetilde{LSp}_4^{c,c^{-1}}.
  \]
  To do this, fix a vector $v$ of weight $\omega_2 - \omega_1$  in $V(\omega_2 - \omega_1) \cong \kk^4[\loopvar^{\pm 1}]$.
  Note that $v$ is extremal in the sense that it is either a highest-weight or lowest-weight vector for each (not necessarily simple) coroot subgroup of $\widetilde{LSp}_4$; the same property then holds automatically for any vector whose weight is conjugate to $\omega_2 - \omega_1$.

  Write $e_i$ and $f_i$ for the standard Chevalley generators of $\widetilde{L\mathfrak{sp}}_4$.
  Since $v$ sits in an $\alpha_1$-root string of length 2, we have
  \[
    x_1(t_1)v = v + t_1e_1v.
  \]
  Similarly, since $e_1v$ has weight $s_1(\omega_2 - \omega_1) = \omega_1 - \omega_2$, it sits in an $\alpha_2$-root string of length 2, and we have
  \[
    x_2(t_2)x_1(t_1)v = v + t_1e_1v + t_1t_2e_2e_1v.
  \]
  Finally, $e_2e_1v$ has weight $s_2s_1(\omega_2 - \omega_1) = \omega_2 - \omega_3$, and we have
  \[
    x_3(t_3)x_2(t_2)x_1(t_1)v = v + t_1e_1v + t_1t_2e_2e_1v + t_1t_2t_3e_3e_2e_1v.
  \]

  Since we know the weights of each component of $x_3(t_3)x_2(t_2)x_1(t_1)v$, we can now compute that
  \[
    hx_3(t_3)x_2(t_2)x_1(t_1)v = h^{\omega_2-\omega_1}v + h^{\omega_1 - \omega_2}t_1e_1v + h^{\omega_2-\omega_3}t_1t_2e_2e_1v + h^{\omega_3-\omega_2}t_1t_2t_3e_3e_2e_1v.
  \]
  From here we can see that the component of $gv$ of weight $\omega_2 - \omega_1$ is
  \[
    (h^{\omega_2-\omega_1} + h^{\omega_1 - \omega_2}t_1t_{\ol{1}} + h^{\omega_2-\omega_3}t_1t_{\ol{1}}t_2t_{\ol{2}} + h^{\omega_3-\omega_2}t_1t_{\ol{1}}t_2t_{\ol{2}}t_3t_{\ol{3}})v.
  \]
  By definition, this is equal to $\Delta_{\omega_2-\omega_1}(g)v$.
  On the other hand, it is exactly what we obtain by using Equation~\eqref{eq:frozen_values} to rewrite the earlier expression for $x_{\omega_2 - \omega_1}$.
\end{proof}

\subsection{The homogeneous element}
We conclude with an extension of Theorem~\ref{th:cluster character equals minor} motivated by the theory of canonical bases.
Cluster monomials form a partial basis of a cluster algebra, and there are various meaningful (and in general different) ways of completing them to a full basis.
One approach in line with the considerations of this paper is to form a basis consisting of cluster characters of potentially non-rigid representations.
By taking the cluster character of a generic representation with each possible $\bfg$-vector, one obtains the generic basis of \cite{Dup12} (which directly generalizes the dual semicanonical basis of Lusztig \cite{Lus00}).

When $Q$ is an acyclic orientation of an $n$-cycle as in Section~\ref{sec:rigidregular}, the simplest example of a non-rigid representation is the representation $\Qrep_\eta$ at the base of a homogeneous tube $\mathcal{T}_\eta$ for $\eta \in \kk^*$.
That is, $\Qrep_\eta$ consists of a one-dimensional vector space at each vertex of $Q$, with arrows acting by isomorphisms such that the monodromy around $Q$ is equal to $\eta$.
The dimension vector of $\Qrep_\eta$ is the primitive imaginary root $\delta$, and if $i_1<\cdots<i_m$ and $o_1<\cdots < o_m$ denote the sinks and sources of $Q$, its $\bfg$-vector is
\[
  \grep(\Qrep_\eta) = \sum_{j=1}^m \omega_{o_j} - \omega_{i_j}.
\]
The cluster character of $\Qrep_\eta$ is independent of $\eta$, and we refer to the resulting element of the generic basis as the homogeneous element.

\begin{theorem}
  \label{thm:homogeneous}
  In $\cA_{\Qdp} \cong \kk[\widetilde{LSL}_n^{c,c^{-1}}]$ the homogeneous element $x_{\Qrep_\eta}$ is the restriction of the level zero minor $\Delta_{\grep(\Qrep_\eta)}$.
\end{theorem}
\begin{proof}
  Similar to Lemma~\ref{lem:regulargvectors}, we have $\grep(\Qrep_\eta) = \omega_S$, where
  \[
    S=[i_1,o_1-1] \cup [i_2,o_2-1] \cup \cdots \cup [i_m,o_m-1] \subset \ZZ_n.
  \]
  The proof now proceeds exactly like that of Theorem~\ref{th:cluster character equals minor}, and we omit the details.
\end{proof}

  Theorem~\ref{thm:homogeneous} has an analogue for an arbitrary symmetric Kac-Moody group of rank two.
  Let $Q$ be the $r$-Kronecker quiver: it has two vertices, and $r$ arrows from vertex 2 to vertex 1.
  Note that all non-initial cluster variables of $\cA_{\Qdp}$ are either preprojective or postinjective.
  However, it contains a homogeneous element $x_M$, where $M$ is any representation with $M_1$ and $M_2$ one-dimensional and at least one arrow acting by an isomorphism.
  Using the obvious injective resolution of $M$ and the fact that its only nontrivial proper submodule is $S_1$, we see that
  \[
    x_M = x_1^{-1}x_2^{r-1}(1 + \hat{y}_1 + \hat{y}_1 \hat{y}_2).
  \]
  One may easily check that such an element is contained in the generic basis of \cite{Dup12}, in the triangular basis of \cite{BZ14}, and the greedy/theta basis of \cite{LLZ14,GHKK14,CGMMRSW15}.

  We have $\cA_{\Qdp} \cong \kk[G^{c,c^{-1}}]$, where $G$ is the Kac-Moody group with Cartan matrix
  \[
    A = \begin{bmatrix} 2 & -r \\ -r & 2 \end{bmatrix}
  \]
  and $c = s_1 s_2$.
  Recall that for any weight $\lambda$ of $G$, there is a nonempty set of isomorphism classes of irreducible $G$-representations for which $\lambda$ is extremal \cite{Kas94}.

  \begin{theorem}
  \label{thm:rank2homogeneous}
  Let $V$ be a $G$-representation for which $\grep(\Qrep) = (r-1) \omega_2 - \omega_1$ is extremal. Then in $\cA_{\Qdp} \cong \kk[G^{c,c^{-1}}]$ the homogeneous element $x_{\Qrep}$ is the restriction of the  minor $\Delta_{(V,\,\grep(\Qrep))}$.
  \end{theorem}

  \begin{proof}
  As usual, we factor a generic element of $G^{c,c^{-1}}$ as 
  \[
    g = x_{\ol{1}}(t_{\ol{1}}) x_{\ol{2}}(t_{\ol{2}})hx_2(t_2)x_1(t_1).
  \]
  By hypothesis, any $v \in V_{\grep(\Qrep)}$ is a lowest-weight vector for the coroot subgroup $\varphi_1(SL_2) \subset G$.
  Denoting the positive Chevalley generators of the Lie algebra of $G$ by $e_1$, $e_2$, we thus have
  \[
    x_1(t_1)v = v + t_1e_1v.
  \]
  Since $v$ is extremal and $e_1v = \ol{s_1}v$ has weight $s_1((r-1) \omega_2 - \omega_1) = \omega_1 - \omega_2$, $e_1v$ is a lowest-weight vector for $\varphi_2(SL_2)$.
  As $v$ is highest-weight for $\varphi_2(SL_2)$, we have
  \[
    x_2(t_2)x_1(t_1)v = v + t_1e_1v + t_1t_2e_2e_1v.
  \]
  Inspecting the weights involved, we see that
  \[
    hx_2(t_2)x_1(t_1)v = h^{(r-1) \omega_2 - \omega_1}v + h^{\omega_1-\omega_2}t_1e_1v + h^{\omega_2 - (r-1)\omega_1}t_1t_2e_2e_1v.
  \]
  Reversing the above analysis we find that the component of $gv$ of weight $(r-1)\omega_2 - \omega_1$ is
  \[
    (h^{(r-1) \omega_2 - \omega_1} + h^{\omega_1-\omega_2}t_1t_{\ol{1}} + h^{\omega_2 - (r-1)\omega_1}t_1t_{\ol{1}}t_2t_{\ol{2}})v.
  \]
  By definition this is $\Delta_{(V,\,\grep(M))}(g)v$.
  On the other hand, using Lemma~\ref{lemma:coefficients_values} and the definition of $\hat{y}_1$, $\hat{y}_2$ one can check that it is also equal to the value of $x_M$ on $g$, hence $x_M$ is equal to the restriction of $\Delta_{(V,\,\grep(M))}$.
  \end{proof}

\bibliographystyle{amsalpha}
\bibliography{bibliography}

\end{document}